\documentclass[10pt,aps,prb, preprint,reqno]{amsart}
\usepackage{amsmath}
\usepackage{amssymb}
\usepackage{hyperref}
  
\newtheorem{theorem}{Theorem}[section]
\newtheorem{remark}{Remark}[section]

\newtheorem{proposition}[theorem]{Proposition}

\begin{document}
\title[Ergodicity of 3D MHD System]{Ergodicity of a Galerkin approximation of three-dimensional magnetohydrodynamics system forced by a degenerate noise}
\author{Kazuo Yamazaki}  
\address{University of Rochester, Department of Mathematics, Rochester, NY, 14627, U.S.A., Phone: 585-275-9427 (kyamazak@ur.rochester.edu)}
\date{}
\maketitle

\begin{abstract}
Magnetohydrodynamics system consists of a coupling of the Navier-Stokes and Maxwell's equations and is most useful in studying the motion of electrically conducting fluids. We prove the existence of a unique invariant, and consequently ergodic, measure for the Galerkin approximation system of the three-dimensional magnetohydrodynamics system. The proof is inspired by those of \cite{EM01, R04} on the Navier-Stokes equations; however, computations involve significantly more complications due to the coupling of the velocity field equations with those of magnetic field that consists of four non-linear terms. 
\vspace{5mm}

\end{abstract}
\footnote{2010MSC : 35Q35, 37L55, 60H15}

\section{Introduction}

Magnetohydrodynamics (MHD) system consists of a coupling of the Navier-Stokes and Maxwell's equations and plays a fundamental role in applied sciences such as astrophysics, geophysics and plasma physics. Ever since the pioneering work of \cite{B50, C51}, its study has been the center of attention from many scientists, in particular engineers, mathematicians and physicists. The Navier-Stokes equations (NSE), due to the lack of coupling which consists of four non-linear terms, is significantly easier to study mathematically. For such systems, an existence of a unique invariant measure would describe a statistical equilibrium to which it approaches. 

In comparison to the stochastic MHD system, the study of the stochastic NSE has a long history, starting from \cite{BT73}. The existence of an invariant measure for the two-dimensional (2D) stochastic NSE was obtained in \cite{F94} (see \cite{DDT94, F97a} for the case of the Burgers' equation and the B$\acute{\mathrm{e}}$nard problem respectively). The uniqueness of such an invariant measure was obtained in \cite{FM95, F97b, F99a}. In \cite{EM01}, E and Mattingly considered the Galerkin approximation of the stochastic NSE in 2D case and identified the minimal set of modes needed in order to guarantee the existence of a unique invariant measure. This result was extended to the 3D stochastic NSE in \cite{R04} (see also \cite{LW04} for the case of the Boussinesq system) (see also \cite{M99, BKL01, HM06}). As we will see, e.g. in Proposition 4.1, the computations in the case of the stochastic MHD system is significantly more complicated due to the four non-linear terms, and ergodicity results in the case of 3D is also more difficult than that of 2D.  Besides the NSE, Burgers' equations, B$\acute{\mathrm{e}}$nard problem, other systems of equations have received much attention concerning ergodicity(e.g. Ginzburg-Landau equations in \cite{EH01, O06}, magnetic B$\acute{\mathrm{e}}$nard problem in \cite{Y16c}, micropolar and magneto-micropolar fluid systems in \cite{Y17b}). 

In contrast, the study of the stochastic MHD system remains relatively incomplete. In \cite{SS99}, Sritharan and Sundar studied the well-posedness of the stochastic MHD system with additive noise, followed by an extension to the case with multiplicative noise by Sango \cite{S10a} and Sundar \cite{S10b}. Large deviation principle type result is obtained in \cite{CM10} but only in 2D and not 3D; similarly, the ergodicity of the stochastic MHD system was proven in \cite{BD07} but only in 2D case. The purpose of this manuscript is to follow the work of \cite{EM01, LW04, R04} and prove the existence of a unique invariant, and consequently ergodic, measure for the Galerkin approximation of the three-dimensional (3D) MHD system forced by a degenerate noise. 

\section{Statement of main result}

We consider for $\displaystyle{x \in \mathbb{T}^{3}}$, the 3D torus, the following MHD system: 
\begin{subequations}\label{1}
\begin{align}
& \frac{d u}{dt} + (u\cdot\nabla) u + \nabla \pi - \nu \Delta u = (b\cdot\nabla) b + \frac{d W_{u}}{d t}, \label{1a}\\
& \frac{d b}{d t} + (u\cdot\nabla) b - \eta \Delta b = (b\cdot\nabla) u + \frac{d W_{b}}{d t}, \label{1b}\\
& \nabla\cdot u = 0, \hspace{3mm} \nabla\cdot b = 0, \label{1c}
\end{align}
\end{subequations}
where $u, b: \mathbb{T}^{3} \times \mathbb{R}^{+} \mapsto \mathbb{R}^{3}$ are the velocity and magnetic vector fields respectively, while $\pi: \mathbb{T}^{3} \times \mathbb{R}^{+} \mapsto \mathbb{R}$ is the pressure scalar field. The constants $\nu, \eta > 0$ represent viscosity and diffusivity, and $W_{u}, W_{b}$ are each an additive white noise to be elaborated shortly. Because the solution $(u,b)$ to the MHD system possesses a rescaling property, namely that $\displaystyle{(u_{\lambda}, b_{\lambda})(x,t) \triangleq \lambda (u,b) (\lambda x, \lambda^{2} t), \lambda \in \mathbb{R}}$, solves the system if $(u,b)(x,t)$ does, we may assume that $\nu = \eta = 1$ throughout the rest of the manuscript. Moreover, we write $\partial_{t}$ to denote $\displaystyle{\frac{\partial}{\partial t} = \frac{d}{dt}}$ and $\displaystyle{\int_{\mathbb{T}^{3}} f(x) dx = \int f}$ for brevity. 

We write in Fourier components, 
\begin{equation}\label{2}
u(x,t) =  \sum_{k \in \mathbb{Z}^{3}} u_{k}(t) e^{ik\cdot x}, \hspace{3mm} b(x,t) = \sum_{k \in \mathbb{Z}^{3}} b_{k}(t) e^{ik \cdot x},
\end{equation}  
where
\begin{equation}
u_{k} \triangleq r_{k} + is_{k}, \hspace{3mm} b_{k} \triangleq \tilde{r}_{k} + i\tilde{s}_{k}, \hspace{3mm} k\cdot r_{k} = k \cdot s_{k} = k \cdot \tilde{r}_{k} = k \cdot \tilde{s}_{k} = 0
\end{equation}
and hence $u_{-k} = \overline{u}_{k}, b_{-k} = \overline{b}_{k}$. Furthermore, setting 
\begin{equation*}
r_{k} \triangleq (r_{k}^{1}, r_{k}^{2}, r_{k}^{3}), \hspace{3mm}  s_{k} \triangleq (s_{k}^{1}, s_{k}^{2}, s_{k}^{3}), \hspace{3mm} \tilde{r}_{k} \triangleq (\tilde{r}_{k}^{1}, \tilde{r}_{k}^{2}, \tilde{r}_{k}^{3}), \hspace{3mm}  \tilde{s}_{k} \triangleq (\tilde{s}_{k}^{1}, \tilde{s}_{k}^{2}, \tilde{s}_{k}^{3}),
\end{equation*}
we see that each $r_{k}^{j}, s_{k}^{j}, \tilde{r}_{k}^{j}, \tilde{s}_{k}^{j}, j = 1, 2, 3, k \in \mathbb{Z}^{3}$, become $\mathbb{R}$-valued. We set 
\begin{subequations}\label{4}
\begin{align}
&dW_{u} \triangleq \sum_{k \in \mathcal{N}} q_{uk}d\beta_{ut}^{k} e^{ik\cdot x}, \hspace{3mm} q_{uk} \triangleq q_{uk}^{r} + i q_{uk}^{s},\\ 
&dW_{b} \triangleq \sum_{k \in \mathcal{N}} q_{bk} d\beta_{bt}^{k} e^{ik\cdot x}, \hspace{4mm} q_{bk} \triangleq q_{bk}^{r} + i q_{bk}^{s}, 
\end{align}
\end{subequations}
where $\{\beta_{ut}^{k} \}_{k}, \{\beta_{bt}^{k}\}_{k}$ are independent 3D Wiener processes, for which we make the following assumptions: the covariance of the noise is diagonal in the Fourier basis and the noise is divergence-free vector valued so that $q_{uk} \cdot k = q_{bk} \cdot k = 0$ for all $k$. Moreover, $\mathcal{N}$ is the set of modes forced, to be elaborated shortly. We recall the projection onto the space of divergence-free vector fields defined by 
\begin{equation}\label{5}
\mathcal{P} (\theta e^{ik\cdot x}) = \left(\theta - \frac{k \otimes k}{\lvert k \rvert^{2}}\cdot \theta \right) e^{ik\cdot x} = \left(\theta - \frac{\theta \cdot k}{\lvert k \rvert^{2}} k \right) e^{ik\cdot x}
\end{equation}
so that the projection gives 
\begin{equation}
\begin{split} 
\mathcal{P} ((u\cdot\nabla) u) 
=& \mathcal{P} \left( \sum_{k \in \mathbb{Z}^{3}} \sum_{h, l: h + l = k} i(u_{h} \cdot l) u_{l} e^{ik\cdot x} \right) \\
=& i \sum_{k \in \mathbb{Z}^{3}} \sum_{h,l: h + l = k} (l\cdot u_{h}) \left( u_{l} - \frac{k\cdot u_{l}}{ \lvert k \rvert^{2}} k \right) e^{ik\cdot x} \\
=& i \sum_{k \in \mathbb{Z}^{3}} \sum_{h, l: h + l = k} (k\cdot u_{h}) \left( u_{l} - \frac{k\cdot u_{l}}{\lvert k \rvert^{2}} k \right) e^{ik\cdot x}, 
\end{split}
\end{equation}
by (\ref{2}), (\ref{5}), and that $h \cdot u_{h} = 0$. Similarly, we may write 
\begin{equation}
\mathcal{P} ((b\cdot\nabla) b) = i \sum_{k \in \mathbb{Z}^{3}} \sum_{h, l: h + l = k} (k\cdot b_{h}) \left( b_{l} - \frac{k\cdot b_{l}}{\lvert k \rvert^{2}} k \right) e^{ik\cdot x},
\end{equation}
\begin{equation}
(u\cdot\nabla) b = i \sum_{k \in \mathbb{Z}^{3}} \sum_{h,l: h + l = k} (k \cdot u_{h}) b_{l} e^{ik\cdot x},  
\end{equation}
\begin{equation}
(b\cdot\nabla) u = i \sum_{k \in \mathbb{Z}^{3}} \sum_{h,l: h + l = k} (k\cdot b_{h}) u_{l}  e^{ik\cdot x}, 
\end{equation}
by that $h\cdot u_{h} = h\cdot b_{h} = 0$. Therefore, writing (\ref{1}) in a differential form after having applied $\mathcal{P}$ gives 
\begin{subequations}
\begin{align}
& du = \left(\Delta u - \mathcal{P} ((u\cdot\nabla) u) + \mathcal{P} ((b\cdot\nabla) b) \right) dt + dW_{u},\\
& db = \left( \Delta b - (u\cdot\nabla) b + (b\cdot\nabla) u \right) dt + dW_{b},
\end{align}
\end{subequations}
where we used the assumption that the noise is divergence-free; hence, for each $k\in \mathbb{Z}^{3}$, 
\begin{equation}
\begin{split}
& du_{k} = ( - \lvert k\rvert^{2} u_{k} - i \sum_{h,l: h + l = k} (k\cdot u_{h}) \left( u_{l} - \frac{k \cdot u_{l}}{\lvert k \rvert^{2}} k \right) \\
& \hspace{10mm} + i\sum_{h,l: h+ l=k} (k\cdot b_{h}) \left( b_{l} - \frac{k\cdot b_{l}}{\lvert k \rvert^{2}} k \right) ) dt + q_{uk} d\beta_{ut}^{k}, 
\end{split}
\end{equation} 
\begin{equation}
 db_{k} = \left( - \lvert k \rvert^{2} b_{k} - i \sum_{h,l: h + l=k} (k\cdot u_{h}) b_{l} + i \sum_{h,l: h + l=k} (k\cdot b_{h}) u_{l}\right) dt + q_{bk} d\beta_{bt}^{k}.
\end{equation} 
We fix $N \in \mathbb{N}$ and set 
\begin{equation}
\mathcal{K}_{N} \triangleq \{k \in \mathbb{Z}^{3}: k \neq (0, 0, 0), \lvert k \rvert_{\infty}  \leq N \},
\end{equation}
where $\lvert k \rvert_{\infty} = \max\{\lvert k_{1} \rvert, \lvert k_{2} \rvert, \lvert k_{3} \rvert \}$, the sup-norm in $\mathbb{R}^{3}$; the exclusion of $(0,0,0)$ is to ensure the mean-zero property. Projecting our system onto the space spanned by $\displaystyle{\{e^{ik\cdot x}\}_{k \in \mathcal{K}_{N}}}$ so that relabeling 
\begin{equation}
u(x,t) = \sum_{k \in \mathcal{K}_{N}} u_{k}(t) e^{ik\cdot x}, \hspace{3mm} b(x,t) = \sum_{k \in \mathcal{K}_{N}}b_{k}(t) e^{ik\cdot x},
\end{equation}
we obtain the following finite-dimensional system: for each $k \in \mathcal{K}_{N}$, 
\begin{equation}\label{15}
\begin{split} 
&du_{k} = ( - \lvert k \rvert^{2} u_{k} - i \sum_{h,l \in \mathcal{K}_{N}: h + l = k} (k\cdot u_{h}) \left( u_{l} - \frac{k \cdot u_{l}}{\lvert k \rvert^{2}} k \right) \\
& \hspace{10mm} + i\sum_{h,l \in \mathcal{K}_{N}: h + l = k} (k\cdot b_{h}) \left( b_{l} - \frac{k \cdot b_{l}}{\lvert k \rvert^{2}} k \right) ) dt + q_{uk} d\beta_{ut}^{k}, 
\end{split} 
\end{equation}
\begin{equation}\label{16}
\begin{split} 
&db_{k} = ( - \lvert k \rvert^{2} b_{k} - i \sum_{h,l \in \mathcal{K}_{N}: h + l = k} (k\cdot u_{h}) b_{l}\\
& \hspace{15mm} + i \sum_{h,l \in \mathcal{K}_{N}: h + l = k} (k\cdot b_{h}) u_{l} ) dt + q_{bk} d\beta_{bt}^{k}.
\end{split} 
\end{equation}
We denote by $\mathcal{N}$ the set of modes forced, i.e. the set of $k \in \mathcal{K}_{N}$ for which $q_{uk} = 0$ and $q_{bk} = 0$ is impossible. Let us now state our main result.  
\begin{theorem}
Suppose $\mathcal{N}$ contains $(1,0,0), (0,1,1), (0,0,1)$. Then there exists a unique invariant measure for the system (\ref{15}), (\ref{16}). 
\end{theorem} 
\begin{remark}
Theorem 1.1 may be extended to the magnetic B$\acute{\mathrm{e}}$nard problem studied in \cite{Y16c} and micropolar and magneto-micropolar fluid systems in \cite{Y17b}; however, it is not clear to the author if it may be extended to the nonhomogeneous MHD system studied in \cite{Y16a}, B$\acute{\mathrm{e}}$nard problem with zero dissipation studied in \cite{Y16b}, or the Hall-MHD system studied in \cite{Y17c}. 

An interesting further improvement of Theorem 2.1 would be to extend the work of \cite{HM06} to the MHD system (\ref{1}). However, the work of \cite{HM06} relied on that of \cite{MP06} and there is a certain symmetry of the velocity equation (\ref{1a}) that is used therein \cite[(3.3), (3.4)]{MP06} that does not seem available for the equation of the magnetic field due to $(u\cdot\nabla) b$ and $(b\cdot\nabla) u$ in (\ref{1b}). In short, Mattingly and Pardoux in \cite{MP06} was able to write the $l$-th mode of the non-linear term $-(u\cdot\nabla) u$ of (\ref{1a}) after applying a curl operator as 
\begin{equation*}
\begin{split}
(-(\mathcal{K} v \cdot \nabla) v)_{l} =&  \sum_{j,k: j + k = l}  ( \frac{ j^{\bot}}{\lvert j \rvert^{2}} v_{j} \cdot k ) v_{k} \\
=& \frac{1}{2} \sum_{j,k: j + k = l} \left( \frac{j^{\bot} \cdot k}{\lvert j \rvert^{2}} v_{j} v_{k} + \frac{k^{\bot} \cdot j}{\lvert k \rvert^{2}} v_{k} v_{j} \right) = \sum_{j,k: j + k =l} c(j,k) v_{j}v_{k}, 
\end{split} 
\end{equation*}
if $v = \nabla \times u$, $\mathcal{K}$ represents the Biot-Savart law operator so that $\mathcal{K} v = u$,  and $c(j,k) = \frac{1}{2} (j^{\bot}\cdot k) (\frac{1}{\lvert j \rvert^{2}} - \frac{1}{\lvert k \rvert^{2}})$. It is not clear how to obtain such a symmetric form for $(u\cdot\nabla) b$ and $(b\cdot\nabla) u$ in (\ref{1b}).  
\end{remark}

\section{Preliminaries}
We denote the standard probability space $(\Omega, \mathcal{F}, \mathcal{F}_{t}, \mathbb{P})$ with $\mathbb{E}$ the expectation with respect to $\mathbb{P}$. For simplicity we write $A \lesssim_{a,b} B, A \approx_{a,b} B$ if there exists a constant $C = C(a,b) \geq 0$ such that $A \leq C B, A = CB$ respectively. As $u_{-k} = \overline{u}_{k}, b_{-k} = \overline{b}_{k}$ due to (3), we consider a smaller set of indices $\tilde{\mathcal{K}}$ for which we denote 
\begin{equation}
\begin{split}
\mathcal{K}_{N}^{1} \triangleq \{k \in \mathbb{Z}^{3}: \lvert k \rvert_{\infty} \leq N, k_{3} > 0 \},
\end{split}
\end{equation}
\begin{equation}
\begin{split}
\mathcal{K}_{N}^{2} \triangleq \{k \in \mathbb{Z}^{3}: \lvert k \rvert_{\infty} \leq N, k_{3} = 0, k_{2} > 0 \}, 
\end{split}
\end{equation}
\begin{equation}
\begin{split}
 \mathcal{K}_{N}^{3} \triangleq \{k \in \mathbb{Z}^{3}: \lvert k \rvert_{\infty} \leq N, k_{3} = k_{2} = 0, k_{1} > 0\},
\end{split}
\end{equation}
and define $\tilde{\mathcal{K}} \triangleq  \mathcal{K}_{N}^{1} \cup \mathcal{K}_{N}^{2} \cup \mathcal{K}_{N}^{3}$. As $\mathcal{K}_{N}$ excludes $(0,0,0)$ due to (13), we see that 
\begin{equation}
\mathcal{K}_{N} = \tilde{\mathcal{K}} \cup (- \tilde{\mathcal{K}}), \hspace{3mm} \tilde{\mathcal{K}} \cap (-\tilde{\mathcal{K}}) = \emptyset. 
\end{equation}
It can be checked that the cardinality of $\displaystyle{\tilde{\mathcal{K}}}$ is $\displaystyle{\#\tilde{\mathcal{K}} = \frac{1}{2} [(2N+1)^{3} - 1]}$ due to $(2N+1)^{2}N$ from $\mathcal{K}_{N}^{1}$, $(2N+1)N$ from $\mathcal{K}_{N}^{2}$, and $N$ from $\mathcal{K}_{N}^{3}$. Hereafter we refer to this number as $D \triangleq$ $\# \tilde{\mathcal{K}} = \frac{1}{2} [(2N+1)^{3} -1]$. Now we denote by 
$\displaystyle{\lVert u \rVert \triangleq  \sum_{k \in \tilde{\mathcal{K}}} \lvert u_{k} \rvert^{2}}$, and see that by Ito's formula using $F(x,t) = x^{2}$ on (\ref{15}), (\ref{16}), and summing over $k \in \tilde{\mathcal{K}}$, we may deduce 
\begin{equation*}
\begin{split} 
&d(\lVert u \rVert^{2} + \lVert b \rVert^{2}) + \sum_{k \in \tilde{\mathcal{K}}}2 \lvert k \rvert^{2} [\lvert u_{k} \rvert^{2} + \lvert b_{k} \rvert^{2}] dt\\
=& \sum_{k \in \tilde{\mathcal{K}}} [2u_{k}\cdot  q_{uk} d\beta_{ut}^{k} + 2b_{k}\cdot  q_{bk} d\beta_{bt}^{k} + \left(Tr (q_{uk}^{T} \overline{q}_{uk}) + Tr (q_{bk}^{T} \overline{q}_{bk}) \right) dt ].
\end{split}
\end{equation*}
We also denote by 
\begin{equation*}
\sigma_{u}^{2} \triangleq \sum_{k \in \tilde{\mathcal{K}}} Tr (q_{uk}^{T} \overline{q}_{uk}), \hspace{3mm} \sigma_{b}^{2} \triangleq \sum_{k \in \tilde{\mathcal{K}}} Tr (q_{bk}^{T} \overline{q}_{bk}),
\end{equation*}
take expected value $\mathbb{E}$ and integrate over $[0,t]$ to obtain 
\begin{equation}
\begin{split} 
& \mathbb{E}[( \lVert u \rVert^{2} + \lVert b \rVert^{2})(t) +  2 \sum_{k \in \tilde{\mathcal{K}}} \lvert k \rvert^{2} \int_{0}^{t} \lvert u_{k} \rvert^{2} + \lvert b_{k} \rvert^{2} ds]\\
=& \mathbb{E} [(\lVert u \rVert^{2} + \lVert b\rVert^{2})(0)] + (\sigma_{u}^{2} + \sigma_{b}^{2})t.
\end{split}
\end{equation} 
As $k \neq (0,0,0)$ in $\tilde{\mathcal{K}}$, we may apply Poincar$\acute{\mathrm{e}}$' inequality to deduce 
\begin{equation*}
\mathbb{E}[( \lVert u \rVert^{2} + \lVert b \rVert^{2})(t) +  2 \int_{0}^{t} \lVert u \rVert^{2} + \lVert b \rVert^{2} ds]
\leq \mathbb{E} [(\lVert u \rVert^{2} + \lVert b\rVert^{2})(0)] + (\sigma_{u}^{2} + \sigma_{b}^{2})t,
\end{equation*}
while Gronwall's inequality type argument also deduces
\begin{equation*}
\mathbb{E}[(\lVert u\rVert^{2} + \lVert b \rVert^{2})(t)]
\leq \mathbb{E} [(\lVert u\rVert^{2} + \lVert b\rVert^{2})(0)] + \frac{(\sigma_{u}^{2} + \sigma_{b}^{2})}{2}. 
\end{equation*}
Now using classical method of applying Krylov-Bogoliubov theorem from \cite{KB37}, specifically \cite[Corollary 11.8]{DZ14}, the existence of an invariant measure follows (see \cite{F94, DDT94, F97a, Y16c} for the cases of the NSE, Burgers equations, B$\acute{\mathrm{e}}$nard and magnetic B$\acute{\mathrm{e}}$nard problems respectively). Because a unique invariant measure in a Polish space is ergodic (see e.g. \cite[Theorem 3.2.6]{DZ96}), it suffices to show the uniqueness to complete the proof of Theorem 2.1. In this regard, we follow the approach of \cite{EM01} and turn to the result of Harris in \cite{H56}: given $\displaystyle{\{x_{n}\}_{n\in \mathbb{Z}^{+} \cup \{0\}}}$, a Markov process on a topological space $\mathbb{X}$ with its Borel $\sigma$-algebra denoted by $\mathcal{B}(\mathbb{X})$, $\{x_{n}\}_{n \in \mathbb{Z}^{+} \cup \{0\}}$ is said to satisfy Harris' condition if there exists a $\sigma$-finite measure $m$ on $\mathbb{X}$ such that if $m(E) > 0, E \in \mathcal{B}(\mathbb{X})$, then 
\begin{equation*}
\mathbb{P}_{x_{0}} \{x_{n} \in E \text{ infinitely often} \} = 1 \hspace{3mm} \forall \hspace{1mm} \text{starting points } x_{0} \in \mathbb{\mathbb{X}}.
\end{equation*} 
Under such a condition, due to Theorem 1 of \cite{H56} there exists a measure $Q$, unique up to a constant multiplier, that solves the equation 
\begin{equation*}
Q(E) = \int_{\mathbb{X}} P(x,E) Q(dx) \hspace{3mm} \forall \hspace{1mm} E \in \mathcal{B}(\mathbb{X})
\end{equation*}
where $P(x, \cdot)$ is the transition probability distribution of the Markov process. Therefore, in the remainder of the manuscript, we devote our effort to show that the transition probability densities are regular by invoking the H$\ddot{\mathrm{o}}$rmander's hypoellipticity condition (\cite{H67}), and that the dynamics enters any neighborhood of the origin infinitely often. 

Our proof is inspired by those of \cite{EM01, R04, LW04, W06} but the computations in the case of the MHD system are significantly more involved due to the coupling of the velocity and magnetic vector fields and four non-linear terms, in particular the proof of Proposition 4.1. For readers' convenience, we give details which were explained only briefly in previous works, while refer to them when proofs are very similar. 

\section{Proof of main result}

\subsection{Smoothness of the transition density}
In this section, we follow the work of \cite{R04} closely (see also \cite{LW04}). By (20), we may write for $k \in \tilde{\mathcal{K}}$, 
\small
\begin{equation}
\sum_{h, l \in \mathcal{K}_{N}: h + l = k } = \sum_{h, l \in \tilde{\mathcal{K}}: h + l = k} + \sum_{h \in \tilde{\mathcal{K}}, l \in - \tilde{\mathcal{K}}: h + l = k} + \sum_{h\in - \tilde{\mathcal{K}}, l \in \tilde{\mathcal{K}}: h + l = k} + \sum_{h,l\in -\tilde{\mathcal{K}}: h + l = k}
\end{equation}
\normalsize 
where using the definitions of $\mathcal{K}_{N}^{1}, \mathcal{K}_{N}^{2}, \mathcal{K}_{N}^{3}$ in (17), (18), (19) respectively, it can be proven that $\{h, l \in -\tilde{\mathcal{K}}: h + l = k \in \tilde{\mathcal{K}} \} = \emptyset$. 

Indeed, as $\displaystyle{\tilde{\mathcal{K}} = \mathcal{K}_{N}^{1} \cup \mathcal{K}_{N}^{2} \cup \mathcal{K}_{N}^{3}}$ due to (17), (18), (19), it suffices to show that when $\displaystyle{h, l \in - \tilde{\mathcal{K}}, h + l = k \notin \mathcal{K}_{N}^{1} \cup\mathcal{K}_{N}^{2} \cup \mathcal{K}_{N}^{3}}$. If 
\begin{equation*}
h, l \in - \tilde{\mathcal{K}} = \mathcal{K}_{N} \setminus \tilde{\mathcal{K}} = \mathcal{K}_{N} \setminus (\mathcal{K}_{N}^{1} \cup \mathcal{K}_{N}^{2} \cup \mathcal{K}_{N}^{3}), 
\end{equation*}
then firstly, $\displaystyle{k_{3} = h_{3} + l_{3} \leq 0}$ as $\displaystyle{h_{3} \leq 0, l_{3} \leq 0}$ and hence $\displaystyle{k \notin \mathcal{K}_{N}^{1}}$. Secondly, if $k_{3} = 0$, then we have $h_{3} = l_{3} = 0$ because $h_{3} \leq 0, l_{3} \leq 0$ which implies $h_{2} \leq 0, l_{2} \leq 0$ as $h, l \in \mathcal{K}_{N} \setminus \mathcal{K}_{N}^{2}$. Therefore, $k_{2} = h_{2} + l_{2} \leq 0$ which implies $k \notin \mathcal{K}_{N}^{2}$. Thirdly, if $k_{3} = k_{2} = 0$, then $h_{3} = l_{3} = 0$ because $h_{3}, l_{3} \leq 0$ as $h, l \in \mathcal{K}_{N} \setminus \mathcal{K}_{N}^{1}$. The facts that $h_{3} = l_{3} = 0$ and $h, l \in \mathcal{K}_{N} \setminus \mathcal{K}_{N}^{2}$ imply $h_{2} \leq 0, l_{2} \leq 0$. The facts that $h_{2} \leq 0, l_{2} \leq 0$ and $k_{2} = 0, k_{2} = h_{2} + l_{2}$ together imply $h_{2} = l_{2} = 0$. That $h_{2} = l_{2} =0, h_{3} = l_{3} = 0$ and $l, h \in \mathcal{K}_{N}\setminus \mathcal{K}_{N}^{3}$ imply $h_{1} \leq 0, l_{1} \leq 0$. Therefore, $k_{1} = h_{1} + l_{1} \leq 0$ which implies $k \notin \mathcal{K}_{N}^{3}$. Hence, we have shown that $\displaystyle{\{h, l \in - \tilde{\mathcal{K}}: h + l = k \in \tilde{\mathcal{K}} \} = \emptyset}$. 

Thus, continuing from (22) we may write for $k \in \tilde{\mathcal{K}}$, denoting by $\sum^{\ast}$ the sum over $h, l \in \tilde{\mathcal{K}}$, 
\begin{equation}
\sum_{h, l \in \mathcal{K}_{N}: h + l = k} 
=\sum_{h+l=k}^{\ast} + \sum_{h-l=k}^{\ast} + \sum_{l-h = k}^{\ast}. 
\end{equation}
Recalling that $u_{-k} = \overline{u}_{k}, b_{-k} = \overline{b}_{k}$ due to (3), we now rewrite (\ref{15}), (\ref{16}) as 
\begin{equation*}
\begin{split} 
du_{k} =& [-\lvert k \rvert^{2} u_{k} - i \sum_{h+l=k}^{\ast} (k\cdot u_{h}) \left( u_{l} - \frac{k\cdot u_{l}}{\lvert k \rvert^{2}} k \right)  -i \sum_{h-l=k}^{\ast} (k\cdot u_{h}) \left( \overline{u}_{l} - \frac{k \cdot \overline{u}_{l}}{\lvert k \rvert^{2}} k \right)\\
 -& i \sum_{l-h=k}^{\ast} (k\cdot \overline{u}_{h}) \left( u_{l} - \frac{k\cdot u_{l}}{ \lvert k \rvert^{2}} k \right) + i \sum_{h+l=k}^{\ast} (k\cdot b_{h}) \left( b_{l} - \frac{k\cdot b_{l}}{\lvert k \rvert^{2}} k \right)\\
 +& i \sum_{h-l=k}^{\ast} (k\cdot b_{h}) \left( \overline{b}_{l} - \frac{k\cdot \overline{b}_{l}}{\lvert k \rvert^{2}} k \right) + i \sum_{l-h=k}^{\ast} (k\cdot \overline{b}_{h}) \left( b_{l} - \frac{k\cdot b_{l}}{\lvert k \rvert^{2}} k \right) ] dt + q_{uk} d\beta_{ut}^{k} 
\end{split} 
\end{equation*}
and 
\begin{equation*}
\begin{split}
db_{k} =& [-\lvert k \rvert^{2} b_{k} - i \sum_{h+l=k}^{\ast} (k\cdot u_{h}) b_{l} - i \sum_{h-l=k}^{\ast} (k\cdot u_{h}) \overline{b}_{l} - i \sum_{l-h=k} (k\cdot \overline{u}_{h}) b_{l} \\
& \hspace{7mm} + i \sum_{h+l=k}^{\ast} (k\cdot b_{h}) u_{l} + i \sum_{h-l=k}^{\ast} (k\cdot b_{h}) \overline{u}_{l} + i \sum_{l-h=k}^{\ast} (k\cdot \overline{b}_{h}) u_{l} ] dt + q_{bk} d\beta_{bt}^{k}.
\end{split}
\end{equation*}
Now by (3), (4a), (4b), 
\begin{equation*}
\begin{split} 
dr_{k} =& [- \lvert k \rvert^{2} r_{k} + \sum_{h+l=k}^{\ast} (k\cdot r_{h}) \left( s_{l} - \frac{k\cdot s_{l}}{\lvert k \rvert^{2}} k\right) + (k\cdot s_{h}) \left( r_{l} - \frac{k\cdot r_{l}}{\lvert k \rvert^{2}} k \right) \\
&- \sum_{h-l=k}^{\ast} (k\cdot r_{h}) \left( s_{l} - \frac{k\cdot s_{l}}{\lvert k \rvert^{2}} k\right) - (k\cdot s_{h}) \left( r_{l} - \frac{k\cdot r_{l}}{\lvert k \rvert^{2}} k \right) \\
&+ \sum_{l-h=k}^{\ast} (k\cdot r_{h}) \left( s_{l} - \frac{k\cdot s_{l}}{\lvert k \rvert^{2}} k\right) - (k\cdot s_{h}) \left( r_{l} - \frac{k\cdot r_{l}}{\lvert k \rvert^{2}} k \right) \\
& - \sum_{h+l=k}^{\ast} (k\cdot \tilde{r}_{h}) \left( \tilde{s}_{l} - \frac{k\cdot \tilde{s}_{l}}{\lvert k \rvert^{2}} k\right) + (k\cdot \tilde{s}_{h}) \left( \tilde{r}_{l} - \frac{k\cdot \tilde{r}_{l}}{\lvert k \rvert^{2}} k \right) \\
&+ \sum_{h-l=k}^{\ast} (k\cdot \tilde{r}_{h}) \left( \tilde{s}_{l} - \frac{k\cdot \tilde{s}_{l}}{\lvert k \rvert^{2}} k\right) - (k\cdot \tilde{s}_{h}) \left( \tilde{r}_{l} - \frac{k\cdot \tilde{r}_{l}}{\lvert k \rvert^{2}} k \right) \\
&- \sum_{l-h=k}^{\ast} (k\cdot \tilde{r}_{h}) \left( \tilde{s}_{l} - \frac{k\cdot \tilde{s}_{l}}{\lvert k \rvert^{2}} k\right) - (k\cdot \tilde{s}_{h}) \left( \tilde{r}_{l} - \frac{k\cdot \tilde{r}_{l}}{\lvert k \rvert^{2}} k \right) ] dt + q_{uk}^{r}d\beta_{ut}^{k},
\end{split} 
\end{equation*}
\begin{equation*}
\begin{split} 
ds_{k} =& [- \lvert k \rvert^{2} s_{k} - \sum_{h+l=k}^{\ast} (k\cdot r_{h}) \left( r_{l} - \frac{k\cdot r_{l}}{\lvert k \rvert^{2}} k\right) - (k\cdot s_{h}) \left( s_{l} - \frac{k\cdot s_{l}}{\lvert k \rvert^{2}} k \right) \\
&- \sum_{h-l=k}^{\ast} (k\cdot r_{h}) \left( r_{l} - \frac{k\cdot r_{l}}{\lvert k \rvert^{2}} k\right) + (k\cdot s_{h}) \left( s_{l} - \frac{k\cdot s_{l}}{\lvert k \rvert^{2}} k \right) \\
&- \sum_{l-h=k}^{\ast} (k\cdot r_{h}) \left( r_{l} - \frac{k\cdot r_{l}}{\lvert k \rvert^{2}} k\right) + (k\cdot s_{h}) \left( s_{l} - \frac{k\cdot s_{l}}{\lvert k \rvert^{2}} k \right) \\
& + \sum_{h+l=k}^{\ast} (k\cdot \tilde{r}_{h}) \left( \tilde{r}_{l} - \frac{k\cdot \tilde{r}_{l}}{\lvert k \rvert^{2}} k\right) - (k\cdot \tilde{s}_{h}) \left( \tilde{s}_{l} - \frac{k\cdot \tilde{s}_{l}}{\lvert k \rvert^{2}} k \right) \\
&+ \sum_{h-l=k}^{\ast} (k\cdot \tilde{r}_{h}) \left( \tilde{r}_{l} - \frac{k\cdot \tilde{r}_{l}}{\lvert k \rvert^{2}} k\right) + (k\cdot \tilde{s}_{h}) \left( \tilde{s}_{l} - \frac{k\cdot \tilde{s}_{l}}{\lvert k \rvert^{2}} k \right) \\
&+ \sum_{l-h=k}^{\ast} (k\cdot \tilde{r}_{h}) \left( \tilde{r}_{l} - \frac{k\cdot \tilde{r}_{l}}{\lvert k \rvert^{2}} k\right) + (k\cdot \tilde{s}_{h}) \left( \tilde{s}_{l} - \frac{k\cdot \tilde{s}_{l}}{\lvert k \rvert^{2}} k \right) ] dt + q_{uk}^{s}d\beta_{ut}^{k},
\end{split} 
\end{equation*}
\begin{equation*}
\begin{split} 
d\tilde{r}_{k} =& [- \lvert k \rvert^{2} \tilde{r}_{k} + \sum_{h+l=k}^{\ast} (k\cdot r_{h}) \tilde{s}_{l} + (k\cdot s_{h}) \tilde{r}_{l} - \sum_{h-l=k}^{\ast} (k\cdot r_{h})\tilde{s}_{l} - (k\cdot s_{h}) \tilde{r}_{l} \\
+& \sum_{l-h=k}^{\ast} (k\cdot r_{h}) \tilde{s}_{l} - (k\cdot s_{h}) \tilde{r}_{l}  - \sum_{h+l=k}^{\ast} (k\cdot \tilde{r}_{h}) s_{l} + (k\cdot \tilde{s}_{h}) r_{l}\\
 +& \sum_{h-l=k}^{\ast} (k\cdot \tilde{r}_{h}) s_{l} - (k\cdot \tilde{s}_{h})r_{l} - \sum_{l-h=k}^{\ast} (k\cdot \tilde{r}_{h}) s_{l} - (k\cdot \tilde{s}_{h}) r_{l}]dt + q_{bk}^{r}d\beta_{bt}^{k},
\end{split} 
\end{equation*}
\begin{equation*}
\begin{split} 
d\tilde{s}_{k} =& [- \lvert k \rvert^{2} \tilde{s}_{k} - \sum_{h+l=k}^{\ast} (k\cdot r_{h}) \tilde{r}_{l} - (k\cdot s_{h}) \tilde{s}_{l} - \sum_{h-l=k}^{\ast} (k\cdot r_{h})\tilde{r}_{l} + (k\cdot s_{h}) \tilde{s}_{l} \\
-& \sum_{l-h=k}^{\ast} (k\cdot r_{h}) \tilde{r}_{l} + (k\cdot s_{h}) \tilde{s}_{l}  + \sum_{h+l=k}^{\ast} (k\cdot \tilde{r}_{h}) r_{l} - (k\cdot \tilde{s}_{h}) s_{l}\\
+& \sum_{h-l=k}^{\ast} (k\cdot \tilde{r}_{h}) r_{l} + (k\cdot \tilde{s}_{h})s_{l} + \sum_{l-h=k}^{\ast} (k\cdot \tilde{r}_{h}) r_{l} + (k\cdot \tilde{s}_{h}) s_{l}]dt + q_{bk}^{s}d\beta_{bt}^{k}.
\end{split} 
\end{equation*}
Accordingly we define for $i = 1, 2, 3$, 
\begin{equation}
\begin{split} 
F_{r_{k}^{i}} \triangleq& - \lvert k \rvert^{2} r_{k}^{i} + \sum_{h+l=k}^{\ast} (k\cdot r_{h}) \left( s_{l}^{i} - \frac{k\cdot s_{l}}{\lvert k \rvert^{2}} k_{i}\right) + (k\cdot s_{h}) \left( r_{l}^{i} - \frac{k\cdot r_{l}}{\lvert k \rvert^{2}} k_{i} \right) \\
&- \sum_{h-l=k}^{\ast} (k\cdot r_{h}) \left( s_{l}^{i} - \frac{k\cdot s_{l}}{\lvert k \rvert^{2}} k_{i}\right) - (k\cdot s_{h}) \left( r_{l}^{i} - \frac{k\cdot r_{l}}{\lvert k \rvert^{2}} k_{i} \right)   \\
&+ \sum_{l-h=k}^{\ast} (k\cdot r_{h}) \left( s_{l}^{i} - \frac{k\cdot s_{l}}{\lvert k \rvert^{2}} k_{i}\right) - (k\cdot s_{h}) \left( r_{l}^{i} - \frac{k\cdot r_{l}}{\lvert k \rvert^{2}} k_{i} \right)  \\
& - \sum_{h+l=k}^{\ast} (k\cdot \tilde{r}_{h}) \left( \tilde{s}_{l}^{i} - \frac{k\cdot \tilde{s}_{l}}{\lvert k \rvert^{2}} k_{i}\right) + (k\cdot \tilde{s}_{h}) \left( \tilde{r}_{l}^{i} - \frac{k\cdot \tilde{r}_{l}}{\lvert k \rvert^{2}} k_{i} \right)  \\
&+ \sum_{h-l=k}^{\ast} (k\cdot \tilde{r}_{h}) \left( \tilde{s}_{l}^{i} - \frac{k\cdot \tilde{s}_{l}}{\lvert k \rvert^{2}} k_{i}\right) - (k\cdot \tilde{s}_{h}) \left( \tilde{r}_{l}^{i} - \frac{k\cdot \tilde{r}_{l}}{\lvert k \rvert^{2}} k_{i} \right)   \\
&- \sum_{l-h=k}^{\ast} (k\cdot \tilde{r}_{h}) \left( \tilde{s}_{l}^{i} - \frac{k\cdot \tilde{s}_{l}}{\lvert k \rvert^{2}} k_{i}\right) - (k\cdot \tilde{s}_{h}) \left( \tilde{r}_{l}^{i} - \frac{k\cdot \tilde{r}_{l}}{\lvert k \rvert^{2}} k_{i} \right),  
\end{split}
\end{equation} 
\begin{equation}
\begin{split} 
F_{s_{k}^{i}} \triangleq & - \lvert k \rvert^{2} s_{k}^{i} - \sum_{h+l=k}^{\ast} (k\cdot r_{h}) \left( r_{l}^{i} - \frac{k\cdot r_{l}}{\lvert k \rvert^{2}} k_{i}\right) - (k\cdot s_{h}) \left( s_{l}^{i} - \frac{k\cdot s_{l}}{\lvert k \rvert^{2}} k_{i} \right) \\
&- \sum_{h-l=k}^{\ast} (k\cdot r_{h}) \left( r_{l}^{i} - \frac{k\cdot r_{l}}{\lvert k \rvert^{2}} k_{i}\right) + (k\cdot s_{h}) \left(s_{l}^{i} - \frac{k\cdot s_{l}}{\lvert k \rvert^{2}} k_{i} \right)   \\
&- \sum_{l-h=k}^{\ast} (k\cdot r_{h}) \left( r_{l}^{i} - \frac{k\cdot r_{l}}{\lvert k \rvert^{2}} k_{i}\right) + (k\cdot s_{h}) \left( s_{l}^{i} - \frac{k\cdot s_{l}}{\lvert k \rvert^{2}} k_{i} \right)   \\
& + \sum_{h+l=k}^{\ast} (k\cdot \tilde{r}_{h}) \left( \tilde{r}_{l}^{i} - \frac{k\cdot \tilde{r}_{l}}{\lvert k \rvert^{2}} k_{i}\right) - (k\cdot \tilde{s}_{h}) \left( \tilde{s}_{l}^{i} - \frac{k\cdot \tilde{s}_{l}}{\lvert k \rvert^{2}} k_{i} \right)  \\
&+ \sum_{h-l=k}^{\ast} (k\cdot \tilde{r}_{h}) \left( \tilde{r}_{l}^{i} - \frac{k\cdot \tilde{r}_{l}}{\lvert k \rvert^{2}} k_{i}\right) + (k\cdot \tilde{s}_{h}) \left( \tilde{s}_{l}^{i} - \frac{k\cdot \tilde{s}_{l}}{\lvert k \rvert^{2}} k_{i} \right)  \\
&+ \sum_{l-h=k}^{\ast} (k\cdot \tilde{r}_{h}) \left( \tilde{r}_{l}^{i} - \frac{k\cdot \tilde{r}_{l}}{\lvert k \rvert^{2}} k_{i}\right) + (k\cdot \tilde{s}_{h}) \left( \tilde{s}_{l}^{i} - \frac{k\cdot \tilde{s}_{l}}{\lvert k \rvert^{2}} k_{i} \right),  
\end{split}
\end{equation} 
\begin{equation}
\begin{split} 
\tilde{F}_{\tilde{r}_{k}^{i}} \triangleq& - \lvert k \rvert^{2} \tilde{r}_{k}^{i} + \sum_{h+l=k}^{\ast} (k\cdot r_{h}) \tilde{s}_{l}^{i} + (k\cdot s_{h}) \tilde{r}_{l}^{i} - \sum_{h-l=k}^{\ast} (k\cdot r_{h})\tilde{s}_{l}^{i} - (k\cdot s_{h}) \tilde{r}_{l}^{i}  \\
&+ \sum_{l-h=k}^{\ast} (k\cdot r_{h}) \tilde{s}_{l}^{i} - (k\cdot s_{h}) \tilde{r}_{l}^{i}  \\
& - \sum_{h+l=k}^{\ast} (k\cdot \tilde{r}_{h}) s_{l}^{i} + (k\cdot \tilde{s}_{h}) r_{l}^{i} + \sum_{h-l=k}^{\ast} (k\cdot \tilde{r}_{h}) s_{l}^{i} - (k\cdot \tilde{s}_{h})r_{l}^{i}  \\
&- \sum_{l-h=k}^{\ast} (k\cdot \tilde{r}_{h}) s_{l}^{i} - (k\cdot \tilde{s}_{h}) r_{l}^{i}, 
\end{split}
\end{equation} 
\begin{equation}
\begin{split} 
\tilde{F}_{\tilde{s}_{k}^{i}} \triangleq& - \lvert k \rvert^{2} \tilde{s}_{k}^{i} - \sum_{h+l=k}^{\ast} (k\cdot r_{h}) \tilde{r}_{l}^{i} - (k\cdot s_{h}) \tilde{s}_{l}^{i} - \sum_{h-l=k}^{\ast} (k\cdot r_{h})\tilde{r}_{l}^{i} + (k\cdot s_{h}) \tilde{s}_{l}^{i} \\
&- \sum_{l-h=k}^{\ast} (k\cdot r_{h}) \tilde{r}_{l}^{i} + (k\cdot s_{h}) \tilde{s}_{l}^{i}  \\
& + \sum_{h+l=k}^{\ast} (k\cdot \tilde{r}_{h}) r_{l}^{i} - (k\cdot \tilde{s}_{h}) s_{l}^{i} + \sum_{h-l=k}^{\ast} (k\cdot \tilde{r}_{h}) r_{l}^{i} + (k\cdot \tilde{s}_{h})s_{l}^{i}  \\
&+ \sum_{l-h=k}^{\ast} (k\cdot \tilde{r}_{h}) r_{l}^{i} + (k\cdot \tilde{s}_{h}) s_{l}^{i},
\end{split}
\end{equation} 
so that we may rewrite 
\begin{equation*}
\begin{split} 
& dr_{k} - F_{r_{k}} (r, s, \tilde{r}, \tilde{s}) dt = q_{uk}^{r} d\beta_{ut}^{k}, \hspace{3mm} ds_{k} - F_{s_{k}} (r, s, \tilde{r}, \tilde{s}) dt = q_{uk}^{s} d\beta_{ut}^{k},\\
& d\tilde{r}_{k} - \tilde{F}_{\tilde{r}_{k}}(r, s, \tilde{r}, \tilde{s}) dt = q_{bk}^{r} d\beta_{bt}^{k}, \hspace{3mm} d\tilde{s}_{k} - \tilde{F}_{\tilde{s}_{k}}(r, s, \tilde{r}, \tilde{s}) dt = q_{bk}^{s} d\beta_{bt}^{k}. 
\end{split} 
\end{equation*}
The solution $(r, \tilde{r}, s, \tilde{s})(t)$ is a Markov process of which its state space is a linear subspace $\displaystyle{U \triangleq \oplus_{k \in \tilde{\mathcal{K}}} (R_{k} \oplus \tilde{\mathcal{R}}_{k} \oplus S_{k} \oplus \tilde{S}_{k}) \subset \mathbb{R}^{12D}}$ where $D = \#\mathcal{K}$ and
\begin{subequations}
\begin{align}
& R_{k} = \{(r, \tilde{r}, s, \tilde{s}) \in \mathbb{R}^{12D}: r_{k} \cdot k = 0, s_{k} = \tilde{r}_{k} = \tilde{s}_{k} = 0,\nonumber\\
& \hspace{40mm} r_{h} = \tilde{r}_{h} = s_{h} = \tilde{s}_{h} = 0, h \neq k \},\\
& S_{k} = \{(r, \tilde{r}, s, \tilde{s}) \in \mathbb{R}^{12D}: s_{k} \cdot k = 0, r_{k} = \tilde{r}_{k} = \tilde{s}_{k} = 0,\nonumber\\
& \hspace{40mm} r_{h} = \tilde{r}_{h} = s_{h} = \tilde{s}_{h} = 0, h \neq k \},\\
& \tilde{R}_{k} = \{(r, \tilde{r}, s, \tilde{s}) \in \mathbb{R}^{12D}: \tilde{r}_{k} \cdot k = 0, r_{k} = s_{k} = \tilde{s}_{k} = 0,\nonumber\\
& \hspace{40mm} r_{h} = \tilde{r}_{h} = s_{h} = \tilde{s}_{h} = 0, h \neq k \},\\
& \tilde{S}_{k} = \{(r, \tilde{r}, s, \tilde{s}) \in \mathbb{R}^{12D}: \tilde{s}_{k} \cdot k = 0, r_{k} = s_{k} = \tilde{r}_{k} = 0, \nonumber\\
& \hspace{40mm} r_{h} = \tilde{r}_{h} = s_{h} = \tilde{s}_{h} = 0, h \neq k \}.
\end{align}
\end{subequations}
We define the Lie algebra $\mathcal{U}$ corresponding to the state space $U$ as 
\begin{equation}
\begin{split} 
\mathcal{U} \triangleq& \{G: G = \sum_{k \in \tilde{\mathcal{K}}} G_{r_{k}^{i}} \frac{\partial}{\partial r_{k}^{i}} + \tilde{G}_{\tilde{r}_{k}^{i}} \frac{\partial}{\partial \tilde{r}_{k}^{i}} + G_{s_{k}^{i}}\frac{\partial}{\partial s_{k}^{i}} + \tilde{G}_{\tilde{s}_{k}^{i}} \frac{\partial}{\partial \tilde{s}_{k}^{i}},\\
& \hspace{25mm} k \cdot G_{r_{k}^{i}} = k \cdot G_{s_{k}^{i}} = k \cdot G_{\tilde{r}_{k}^{i}} = k \cdot G_{\tilde{s}_{k}^{i}} = 0 \}
\end{split}
\end{equation}
and its subspace of constant vector fields $\mathcal{U}_{k} \triangleq \mathcal{R}_{k} \oplus \tilde{\mathcal{R}}_{k} \oplus \mathcal{S}_{k} \oplus \tilde{\mathcal{S}}_{k} \subset \mathcal{U}$ where 
\begin{equation*}
\begin{split} 
& \mathcal{R}_{k} \triangleq \{\sum_{i=1}^{3} r_{k}^{i} \frac{\partial}{\partial r_{k}^{i}}: r_{k} \cdot k = 0 \}, \hspace{3mm} \mathcal{S}_{k} \triangleq \{\sum_{i=1}^{3} s_{k}^{i} \frac{\partial}{\partial s_{k}^{i}}: s_{k} \cdot k = 0 \},\\
& \tilde{\mathcal{R}}_{k} \triangleq \{\sum_{i=1}^{3} \tilde{r}_{k}^{i} \frac{\partial}{\partial \tilde{r}_{k}^{i}}: \tilde{r}_{k} \cdot k = 0 \}, \hspace{3mm} \tilde{\mathcal{S}}_{k} \triangleq \{\sum_{i=1}^{3} \tilde{s}_{k}^{i} \frac{\partial}{\partial \tilde{s}_{k}^{i}}: \tilde{s}_{k} \cdot k = 0 \}.
\end{split} 
\end{equation*}
We now derive sufficient conditions on the set $\mathcal{N}$ of forced modes so that the algebra generated by the vector fields $\{F_{0}\} \cup \mathcal{U}_{k}, k \in \mathcal{N}$, where 
\begin{equation}
F_{0} \triangleq \sum_{k \in \tilde{\mathcal{K}}: i = 1, 2, 3} F_{r_{k}^{i}} \frac{\partial}{\partial r_{k}^{i}} + F_{s_{k}^{i}} \frac{\partial}{\partial s_{k}^{i}} + \tilde{F}_{\tilde{r}_{k}^{i}} \frac{\partial}{\partial \tilde{r}_{k}^{i}} + \tilde{F}_{\tilde{s}_{k}^{i}} \frac{\partial}{\partial \tilde{s}_{k}^{i}},
\end{equation}
contains all constant vector fields of $\mathcal{U}$, which implies that the H$\ddot{\mathrm{o}}$rmander's condition for hypoellipticity is satisfied. Thus, we now need to compute the Lie brackets of the form $[[F_{0}, V], W]$ for general constant vector fields $V, W$. 

\begin{proposition}
Suppose $m, n \in \tilde{\mathcal{K}}, V \in \mathcal{U}_{m}, W \in \mathcal{U}_{n}$ where 
\begin{equation}
V \triangleq \sum_{j=1}^{3} v_{j}^{r} \frac{\partial}{\partial r_{m}^{j}} + v_{j}^{s} \frac{\partial}{\partial s_{m}^{j}} + \tilde{v}_{j}^{r} \frac{\partial}{\partial \tilde{r}_{m}^{j}} + \tilde{v}_{j}^{s} \frac{\partial}{\partial \tilde{s}_{m}^{j}},
\end{equation}
\begin{equation}
 W \triangleq \sum_{l=1}^{3} w_{l}^{r} \frac{\partial}{\partial r_{n}^{l}} + w_{l}^{s} \frac{\partial}{\partial s_{n}^{l}} + \tilde{w}_{l}^{r} \frac{\partial}{\partial \tilde{r}_{n}^{l}} + \tilde{w}_{l}^{s} \frac{\partial}{\partial \tilde{s}_{n}^{l}}. 
\end{equation}
If $k = m +n, h = n- m, g = m- n$, then 
\small
\begin{equation}
\begin{split} 
&[[F_{0}, V], W]\\
=& [(v^{s} \cdot k) P_{k} (w^{r}) + (w^{r} \cdot k) P_{k}(v^{s}) + (v^{r} \cdot k) P_{k}(w^{s}) + (w^{s}\cdot k) P_{k}(v^{r})\\
&- (\tilde{v}^{s}\cdot k) P_{k}(\tilde{w}^{r}) - (\tilde{w}^{r}\cdot k) P_{k}(\tilde{v}^{s}) - (\tilde{v}^{r}\cdot k) P_{k}(\tilde{w}^{s}) - (\tilde{w}^{s}\cdot k) P_{k}(\tilde{v}^{r})] \cdot \frac{\partial}{\partial r_{k}}\\
&+[-(v^{r} \cdot k) P_{k} (w^{r}) - (w^{r} \cdot k) P_{k}(v^{r}) + (v^{s} \cdot k) P_{k}(w^{s}) + (w^{s}\cdot k) P_{k}(v^{s}) \\
&+ (\tilde{v}^{r}\cdot k) P_{k}(\tilde{w}^{r}) + (\tilde{w}^{r}\cdot k) P_{k}(\tilde{v}^{r}) - (\tilde{v}^{s}\cdot k) P_{k}(\tilde{w}^{s}) - (\tilde{w}^{s}\cdot k) P_{k}(\tilde{v}^{s})] \cdot \frac{\partial}{\partial s_{k}}  \\
&+[-(v^{s} \cdot h) P_{h} (w^{r}) - (w^{r} \cdot h) P_{h}(v^{s}) + (v^{r} \cdot h) P_{h}(w^{s}) + (w^{s}\cdot h) P_{h}(v^{r})\\
&+ (\tilde{v}^{s}\cdot h) P_{h}(\tilde{w}^{r}) + (\tilde{w}^{r}\cdot h) P_{h}(\tilde{v}^{s}) - (\tilde{v}^{r}\cdot h) P_{h}(\tilde{w}^{s}) - (\tilde{w}^{s}\cdot h) P_{h}(\tilde{v}^{r})] \cdot \frac{\partial}{\partial r_{h}}  \\
&+[-(v^{r} \cdot h) P_{h} (w^{r}) - (w^{r} \cdot h) P_{h}(v^{r}) - (v^{s} \cdot h) P_{h}(w^{s}) - (w^{s}\cdot h) P_{h}(v^{s}) \\
&+ (\tilde{v}^{r}\cdot h) P_{h}(\tilde{w}^{r}) + (\tilde{w}^{r}\cdot h) P_{h}(\tilde{v}^{r}) + (\tilde{v}^{s}\cdot h) P_{h}(\tilde{w}^{s}) + (\tilde{w}^{s}\cdot h) P_{h}(\tilde{v}^{s})] \cdot \frac{\partial}{\partial s_{h}} \\
&+[(v^{s} \cdot g) P_{g} (w^{r}) + (w^{r} \cdot g) P_{g}(v^{s}) - (v^{r} \cdot g) P_{g}(w^{s}) - (w^{s}\cdot g) P_{g}(v^{r}) \\
&- (\tilde{v}^{s}\cdot g) P_{g}(\tilde{w}^{r}) - (\tilde{w}^{r}\cdot g) P_{g}(\tilde{v}^{s}) + (\tilde{v}^{r}\cdot g) P_{g}(\tilde{w}^{s}) + (\tilde{w}^{s}\cdot g) P_{g}(\tilde{v}^{r})] \cdot \frac{\partial}{\partial r_{g}}  \\
&+[-(v^{r} \cdot g) P_{g} (w^{r}) - (w^{r} \cdot g) P_{g}(v^{r}) - (v^{s} \cdot g) P_{g}(w^{s}) - (w^{s}\cdot g) P_{g}(v^{s}) \\
&+ (\tilde{v}^{r}\cdot g) P_{g}(\tilde{w}^{r}) + (\tilde{w}^{r}\cdot g) P_{g}(\tilde{v}^{r}) + (\tilde{v}^{s}\cdot g) P_{g}(\tilde{w}^{s}) + (\tilde{w}^{s}\cdot g) P_{g}(\tilde{v}^{s})] \cdot \frac{\partial}{\partial s_{g}} \\
&+ [-(\tilde{v}^{s}\cdot k) w^{r} + (w^{r}\cdot k) \tilde{v}^{s} - (\tilde{v}^{r} \cdot k) w^{s} + (w^{s} \cdot k) \tilde{v}^{r} \\
&+ (v^{s}\cdot k) \tilde{w}^{r} - (\tilde{w}^{r}\cdot k) v^{s} + (v^{r} \cdot k) \tilde{w}^{s} - (\tilde{w}^{s} \cdot k) v^{r}] \cdot \frac{\partial}{\partial \tilde{r}_{k}} \\
&+ [(\tilde{v}^{r}\cdot k) w^{r} - (w^{r}\cdot k) \tilde{v}^{r} - (\tilde{v}^{s} \cdot k) w^{s} + (w^{s} \cdot k) \tilde{v}^{s} \\
&- (v^{r}\cdot k) \tilde{w}^{r} + (\tilde{w}^{r}\cdot k) v^{r} + (v^{s} \cdot k) \tilde{w}^{s} - (\tilde{w}^{s} \cdot k) v^{s}] \cdot \frac{\partial}{\partial \tilde{s}_{k}} \\
&+ [(\tilde{v}^{s}\cdot h) w^{r} - (w^{r}\cdot h) \tilde{v}^{s} - (\tilde{v}^{r} \cdot h) w^{s} + (w^{s} \cdot h) \tilde{v}^{r} \\
&- (v^{s}\cdot h) \tilde{w}^{r} + (\tilde{w}^{r}\cdot h) v^{s} + (v^{r} \cdot h) \tilde{w}^{s} - (\tilde{w}^{s} \cdot h) v^{r}] \cdot \frac{\partial}{\partial \tilde{r}_{h}}\\
&+ [(\tilde{v}^{r}\cdot h) w^{r} - (w^{r}\cdot h) \tilde{v}^{r} + (\tilde{v}^{s} \cdot h) w^{s} - (w^{s} \cdot h) \tilde{v}^{s} \\
&- (v^{r}\cdot h) \tilde{w}^{r} + (\tilde{w}^{r}\cdot h) v^{r} - (v^{s} \cdot h) \tilde{w}^{s} + (\tilde{w}^{s} \cdot h) v^{s}] \cdot \frac{\partial}{\partial \tilde{s}_{h}}\\
&+ [-(\tilde{v}^{s}\cdot g) w^{r} + (w^{r}\cdot g) \tilde{v}^{s} + (\tilde{v}^{r} \cdot g) w^{s} - (w^{s} \cdot g) \tilde{v}^{r} \\
&+ (v^{s}\cdot g) \tilde{w}^{r} - (\tilde{w}^{r}\cdot g) v^{s} - (v^{r} \cdot g) \tilde{w}^{s} + (\tilde{w}^{s} \cdot g) v^{r}] \cdot \frac{\partial}{\partial \tilde{r}_{g}} \\
&+ [(\tilde{v}^{r}\cdot g) w^{r} - (w^{r}\cdot g) \tilde{v}^{r} + (\tilde{v}^{s} \cdot g) w^{s} - (w^{s} \cdot g) \tilde{v}^{s} \\
&- (v^{r}\cdot g) \tilde{w}^{r} + (\tilde{w}^{r}\cdot g) v^{r} - (v^{s} \cdot g) \tilde{w}^{s} + (\tilde{w}^{s} \cdot g) v^{s}] \cdot \frac{\partial}{\partial \tilde{s}_{g}}   
\end{split}
\end{equation} 
\normalsize 
where $\displaystyle{P_{k}(v)_{i} \triangleq v_{i} - \frac{k_{i}}{\lvert k \rvert^{2}} (v\cdot k)}$ is the projection onto the space orthogonal to the vector $k$, and the terms corresponding to indices not in $\tilde{\mathcal{K}}$ are zero. 
\end{proposition}

\begin{proof}
We first compute the derivatives of the components of $F_{0}$ in (30) using (24), (25), (26), (27): 
\begin{equation}
\begin{split} 
\frac{\partial F_{r_{k}^{i}}}{\partial r_{m}^{j}} 
=& -\lvert k \rvert^{2} \delta_{i,j} \delta_{k,m} + k_{j}\left(s_{k-m}^{i} - \frac{k\cdot s_{k-m}}{\lvert k \rvert^{2}} k_{i}\right) + k\cdot s_{k-m}\left( \delta_{i,j} - \frac{k_{j}k_{i}}{\lvert k \rvert^{2}}\right) \\
&- k_{j}\left(s_{m-k}^{i} - \frac{k\cdot s_{m-k}}{\lvert k \rvert^{2}} k_{i}\right) + k\cdot s_{k+m}\left( \delta_{i,j} - \frac{k_{j}k_{i}}{\lvert k \rvert^{2}}\right)  \\
&+ k_{j}\left(s_{k+m}^{i} - \frac{k\cdot s_{k+m}}{\lvert k \rvert^{2}} k_{i}\right) - k\cdot s_{m-k}\left( \delta_{i,j} - \frac{k_{j}k_{i}}{\lvert k \rvert^{2}}\right) \\
=& -\lvert k \rvert^{2} \delta_{i,j} \delta_{k,m} + k_{j}(s_{k-m}^{i} - s_{m-k}^{i} + s_{k+m}^{i}) \\
&+ k\cdot (s_{k-m} - s_{m-k} + s_{k+m}) \left(\delta_{i,j} - 2\frac{k_{i}k_{j}}{\lvert k \rvert^{2}}\right),
\end{split}
\end{equation} 
where e.g. 
\begin{equation*}
\delta_{i,j} = \begin{cases}
1 & \text{ if } i = j, \\
0 & \text{ if } i \neq j.
\end{cases}
\end{equation*}
Similarly we may compute 
\begin{equation}
\begin{split} 
\frac{\partial F_{r_{k}^{i}}}{\partial s_{m}^{j}} =& k_{j} (r_{k-m}^{i} + r_{m-k}^{i} - r_{m+k}^{i})\\
&+ k\cdot (r_{k-m} + r_{m-k}- r_{k+m}) \left(\delta_{i,j} - 2 \frac{k_{i}k_{j}}{\lvert k\rvert^{2}} \right), 
\end{split}
\end{equation} 
\begin{equation}
\begin{split} 
\frac{\partial F_{r_{k}^{i}}}{\partial \tilde{r}_{m}^{j}} =& k_{j}(-\tilde{s}_{k-m}^{i} + \tilde{s}_{m-k}^{i} - \tilde{s}_{m+k}^{i})\\
& + k\cdot (-\tilde{s}_{k-m} + \tilde{s}_{m-k} - \tilde{s}_{k+m})\left(\delta_{i,j} - 2 \frac{k_{i}k_{j}}{\lvert k\rvert^{2}}\right), 
\end{split}
\end{equation} 
\begin{equation}
\begin{split} 
\frac{\partial F_{r_{k}^{i}}}{\partial \tilde{s}_{m}^{j}} =& k_{j} (-\tilde{r}_{k-m}^{i} - \tilde{r}_{m-k}^{i} + \tilde{r}_{m+k}^{i}) \\
&+ k\cdot (-\tilde{r}_{k-m} - \tilde{r}_{m-k} + \tilde{r}_{k+m}) \left(\delta_{i,j} - 2 \frac{k_{i}k_{j}}{\lvert k\vert^{2}}\right),  
\end{split}
\end{equation} 
\begin{equation}
\begin{split} 
\frac{\partial F_{s_{k}^{i}}}{\partial r_{m}^{j}} 
=& k_{j}(-r_{k-m}^{i} - r_{m-k}^{i} - r_{m+k}^{i})\\
&+ k\cdot (-r_{k-m} - r_{m-k} - r_{m+k} ) \left(\delta_{i,j} - 2 \frac{k_{i}k_{j}}{\lvert k\rvert^{2}}\right), 
\end{split}
\end{equation} 
\begin{equation}
\begin{split} 
\frac{\partial F_{s_{k}^{i}}}{\partial s_{m}^{j}} =& - \lvert k\rvert^{2} \delta_{i,j}\delta_{k,m} + k_{j}(s_{k-m}^{i} - s_{m-k}^{i} - s_{m+k}^{i})\\
&+ k \cdot (s_{k-m} - s_{m-k} - s_{k+m}) \left(\delta_{i,j} - 2 \frac{k_{i}k_{j}}{\lvert k\rvert^{2}} \right),  
\end{split}
\end{equation} 
\begin{equation}
\begin{split} 
\frac{\partial F_{s_{k}^{i}}}{\partial \tilde{r}_{m}^{j}} 
=& k_{j} (\tilde{r}_{k-m}^{i} + \tilde{r}_{m-k}^{i} + \tilde{r}_{m+k}^{i})\\
& + k \cdot (\tilde{r}_{k-m}+ \tilde{r}_{m-k} + \tilde{r}_{m+k} ) \left(\delta_{i,j} - 2 \frac{k_{j}k_{i}}{\lvert k\rvert^{2}} \right), 
\end{split}
\end{equation} 
\begin{equation}
\begin{split}
\frac{\partial F_{s_{k}^{i}}}{\partial \tilde{s}_{m}^{j}} =& k_{j} (-\tilde{s}_{k-m}^{i} + \tilde{s}_{m-k}^{i} + \tilde{s}_{m+k}^{i})\\
& + k \cdot (-\tilde{s}_{k-m}+ \tilde{s}_{m-k} + \tilde{s}_{m+k} ) \left(\delta_{i,j} - 2 \frac{k_{j}k_{i}}{\lvert k\rvert^{2}} \right),  
\end{split}
\end{equation} 
\begin{equation}
\frac{\partial \tilde{F}_{\tilde{r}_{k}^{i}}}{\partial r_{m}^{j}} 
= k_{j} (\tilde{s}_{k-m}^{i} - \tilde{s}_{m-k}^{i} + \tilde{s}_{m+k}^{i}) + k\cdot (-\tilde{s}_{k-m} + \tilde{s}_{m-k} - \tilde{s}_{m+k})\delta_{i,j},
\end{equation}
\begin{equation}
\frac{\partial \tilde{F}_{\tilde{r}_{k}^{i}}}{\partial s_{m}^{j}} = k_{j} (\tilde{r}_{k-m}^{i} + \tilde{r}_{m-k}^{i} - \tilde{r}_{m+k}^{i}) + k\cdot (-\tilde{r}_{k-m}- \tilde{r}_{m-k} + \tilde{r}_{m+k} )\delta_{i,j}, 
\end{equation}
\begin{equation}
\begin{split} 
\frac{\partial \tilde{F}_{\tilde{r}_{k}^{i}}}{\partial \tilde{r}_{m}^{j}} 
=& -\lvert k\rvert^{2} \delta_{i,j}\delta_{k,m}\\
&+ k\cdot (s_{k-m} - s_{m-k} + s_{k+m}) \delta_{i,j} + k_{j}(-s_{k-m}^{i} + s_{m-k}^{i} - s_{m+k}^{i}),
\end{split}
\end{equation} 
\begin{equation}
\frac{\partial \tilde{F}_{\tilde{r}_{k}^{i}}}{\partial \tilde{s}_{m}^{j}} = k\cdot (r_{k-m} + r_{m-k}- r_{k+m}) \delta_{i,j} + k_{j}(-r_{k-m}^{i} - r_{m-k}^{i} + r_{m+k}^{i}), 
\end{equation}
\begin{equation}
\frac{\partial \tilde{F}_{\tilde{s}_{k}^{i}}}{\partial r_{m}^{j}} = k_{j} (-\tilde{r}_{k-m}^{i} - \tilde{r}_{m-k}^{i} - \tilde{r}_{m+k}^{i}) + k\cdot (\tilde{r}_{k-m}+ \tilde{r}_{m-k} + \tilde{r}_{k+m} )\delta_{i,j}, 
\end{equation}
\begin{equation}
\frac{\partial \tilde{F}_{\tilde{s}_{k}^{i}}}{\partial s_{m}^{j}} = k_{j} (\tilde{s}_{k-m}^{i} - \tilde{s}_{m-k}^{i} - \tilde{s}_{m+k}^{i}) + k\cdot (-\tilde{s}_{k-m} + \tilde{s}_{m-k}+ \tilde{s}_{k+m} )\delta_{i,j}, 
\end{equation}
\begin{equation}
\frac{\partial \tilde{F}_{\tilde{s}_{k}^{i}}}{\partial \tilde{r}_{m}^{j}} = k\cdot (-r_{k-m} - r_{m-k}- r_{k+m} )\delta_{i,j}  + k_{j} (r_{k-m}^{i} + r_{m-k}^{i} + r_{m+k}^{i}), 
\end{equation}
\begin{equation}
\begin{split} 
\frac{\partial \tilde{F}_{\tilde{s}_{k}^{i}}}{\partial \tilde{s}_{m}^{j}} =& -\lvert k\rvert^{2} \delta_{i,j} \delta_{k,m}\\
& + k\cdot (s_{k-m}  - s_{m-k}- s_{k+m})\delta_{i,j}  + k_{j} (-s_{k-m}^{i} + s_{m-k}^{i} + s_{m+k}^{i}).  
\end{split}
\end{equation}

Next, we must compute the second derivatives: denoting $\displaystyle{\alpha_{j,l}(i,k) \triangleq k_{j} \delta_{i,l}+ k_{l} \delta_{i,j}}$ $\displaystyle{ - 2 \frac{k_{i}k_{j}k_{l}}{\lvert k\rvert^{2}}}$, by (34), 
\begin{equation*}
\frac{\partial^{2} F_{r_{k}^{i}}}{\partial r_{m}^{j} \partial r_{n}^{l}} = \frac{\partial^{2} F_{r_{k}^{i}}}{\partial r_{m}^{j} \partial \tilde{r}_{n}^{l}} =\frac{\partial^{2} F_{r_{k}^{i}}}{\partial r_{m}^{j} \partial \tilde{s}_{n}^{l}}=0,  
\end{equation*}
\begin{equation}
\begin{split} 
\frac{\partial^{2} F_{r_{k}^{i}}}{\partial r_{m}^{j}\partial s_{n}^{l}} =& k_{j} (\delta_{i,l} \delta_{n, k-m} - \delta_{i,l} \delta_{n, m-k} + \delta_{i,l} \delta_{n, k+m}) \\
&+ k_{l} (\delta_{n, k-m} - \delta_{n, m-k} + \delta_{n, k+m}) \left(\delta_{i,j} - 2 \frac{k_{i}k_{j}}{\lvert k \rvert^{2}} \right)\\
=& (\delta_{n, k-m} - \delta_{n, m-k} + \delta_{n, k+m}) \alpha_{j,l}(i,k).
\end{split}
\end{equation} 
Similarly, using equations (35)-(49), we may deduce 
\begin{equation}
\frac{\partial^{2} F_{r_{k}^{i}}}{\partial r_{n}^{l} \partial s_{m}^{j}} 
= (\delta_{n, k-m} + \delta_{n, m-k} - \delta_{n, m+k}) \alpha_{j,l}(i,k),
\end{equation}
\begin{equation}
 \frac{\partial^{2} F_{r_{k}^{i}}}{\partial \tilde{r}_{m}^{j}\partial \tilde{s}_{n}^{l}} 
= (-\delta_{n, k-m} + \delta_{n, m-k} - \delta_{n, m+k}) \alpha_{j,l}(i,k),
\end{equation}
\begin{equation}
\frac{\partial^{2} F_{r_{k}^{i}}}{\partial \tilde{r}_{n}^{l} \partial \tilde{s}_{m}^{j}}  = (-\delta_{n, k-m} - \delta_{n, m-k} + \delta_{n, m+k}) \alpha_{j,l}(i,k),
\end{equation}
\begin{equation}
 \frac{\partial^{2} F_{s_{k}^{i}}}{\partial r_{m}^{j} \partial r_{n}^{l}} 
= (-\delta_{n, k-m} - \delta_{n, m-k} - \delta_{n, m+k}) \alpha_{j,l}(i,k),
\end{equation}
\begin{equation}
 \frac{\partial^{2} F_{s_{k}^{i}}}{\partial s_{m}^{j} \partial s_{n}^{l}} 
= (\delta_{n, k-m} - \delta_{n, m-k} - \delta_{n, m+k}) \alpha_{j,l}(i,k),
\end{equation}
\begin{equation}
\frac{\partial^{2} F_{s_{k}^{i}}}{\partial \tilde{r}_{m}^{j} \partial \tilde{r}_{n}^{l}} = (\delta_{n, k-m} + \delta_{n, m-k} + \delta_{n, m+k}) \alpha_{j,l}(i,k),
\end{equation}
\begin{equation}
\frac{\partial^{2} F_{s_{k}^{i}}}{\partial \tilde{s}_{m}^{j} \partial \tilde{s}_{n}^{l}} 
 = (-\delta_{n, k-m} + \delta_{n, m-k} + \delta_{n, m+k}) \alpha_{j,l}(i,k), 
\end{equation}
\begin{equation}
 \frac{\partial^{2} \tilde{F}_{\tilde{r}_{k}^{i}}}{\partial r_{m}^{j} \partial \tilde{s}_{n}^{l}}= (\delta_{n, k-m} - \delta_{n, m-k} + \delta_{n, m+k}) (k_{j}\delta_{i,l} - k_{l} \delta_{i,j}),
\end{equation}
\begin{equation}
\frac{\partial^{2} \tilde{F}_{\tilde{r}_{k}^{i}}}{\partial r_{n}^{l} \partial \tilde{s}_{m}^{j}} = (\delta_{n, k-m} + \delta_{n, m-k}- \delta_{n, k+m} ) (k_{l} \delta_{i,j}  - k_{j}\delta_{i,l}),
\end{equation}
\begin{equation}
 \frac{\partial^{2} \tilde{F}_{\tilde{r}_{k}^{i}}}{\partial s_{m}^{j} \partial \tilde{r}_{n}^{l}} = (\delta_{n, k-m} + \delta_{n, m-k} - \delta_{n, m+k}) (k_{j}\delta_{i,l} - k_{l} \delta_{i,j}),
\end{equation}
\begin{equation}
 \frac{\partial^{2} \tilde{F}_{\tilde{r}_{k}^{i}}}{\partial s_{n}^{l} \partial \tilde{r}_{m}^{j}} 
= (\delta_{n, k-m} - \delta_{n, m-k}+ \delta_{n, k+m} )( k_{l} \delta_{i,j}  -  k_{j} \delta_{i,l}), 
\end{equation}
\begin{equation}
 \frac{\partial^{2} \tilde{F}_{\tilde{s}_{k}^{i}}}{\partial r_{m}^{j} \partial \tilde{r}_{n}^{l}} 
 = (-\delta_{n, k-m} - \delta_{n, m-k} - \delta_{n, m+k}) (k_{j} \delta_{i,l} - k_{l} \delta_{i,j}),
\end{equation}
\begin{equation}
 \frac{\partial^{2} \tilde{F}_{\tilde{s}_{k}^{i}}}{\partial r_{n}^{l} \partial \tilde{r}_{m}^{j}} = (\delta_{n, k-m} + \delta_{n, m-k}+ \delta_{n, k+m} ) (  -  k_{l}  \delta_{i,j}+ k_{j} \delta_{i,l}),
\end{equation}
\begin{equation}
 \frac{\partial^{2} \tilde{F}_{\tilde{s}_{k}^{i}}}{\partial s_{m}^{j} \partial \tilde{s}_{n}^{l}} = (\delta_{n, k-m} - \delta_{n, m-k} - \delta_{n, m+k}) (k_{j} \delta_{i,l} - k_{l} \delta_{i,j}),
\end{equation}
\begin{equation}
 \frac{\partial^{2} \tilde{F}_{\tilde{s}_{k}^{i}}}{\partial s_{n}^{l} \partial \tilde{s}_{m}^{j}} = (\delta_{n, k-m} - \delta_{n, m-k}- \delta_{n, k+m} )(k_{l} \delta_{i,j} - k_{j}\delta_{i,l}).
\end{equation}
It may be computed similarly to show that all other partial derivatives of $F_{r_{k}^{i}}, F_{s_{k}^{i}}, \tilde{F}_{\tilde{r}_{k}^{i}}$ and $\tilde{F}_{\tilde{s}_{k}^{i}}$ all vanish. Making use of these computations of second derivatives, in particular those that vanish, we may compute $[[F_{0}, V], W]$ to reach that it only consists of 
\begin{equation}
\begin{split} 
[[F_{0}, V], W]
=& \sum_{k \in \tilde{\mathcal{K}}}\sum_{i,j,l=1}^{3} [w_{l}^{r} v_{j}^{s} (\frac{\partial^{2} F_{r_{k}^{i}}}{\partial r_{n}^{l} \partial s_{m}^{j}}) + w_{l}^{s} v_{j}^{r} (\frac{\partial^{2} F_{r_{k}^{i}}}{\partial s_{n}^{l} \partial r_{m}^{j}})\\
& \hspace{12mm} + \tilde{w}_{l}^{r} \tilde{v}_{j}^{s} (\frac{\partial^{2} F_{r_{k}^{i}}}{ \partial \tilde{r}_{n}^{l} \partial \tilde{s}_{m}^{j}  }) + \tilde{w}_{l}^{s} \tilde{v}_{j}^{r}(\frac{\partial^{2} F_{r_{k}^{i}}}{\partial \tilde{s}_{n}^{l}\partial \tilde{r}_{m}^{j} })] \frac{\partial}{\partial r_{k}^{i}}  \\
& \hspace{10mm} + [w_{l}^{r} v_{j}^{r} (\frac{\partial^{2} F_{s_{k}^{i}}}{\partial r_{n}^{l} \partial r_{m}^{j}}) + w_{l}^{s} v_{j}^{s} (\frac{\partial^{2} F_{s_{k}^{i}}}{\partial s_{n}^{l} \partial s_{m}^{j}})\\
& \hspace{12mm} + \tilde{w}_{l}^{r} \tilde{v}_{j}^{r} (\frac{\partial^{2} F_{s_{k}^{i}}}{ \partial \tilde{r}_{n}^{l}  \partial \tilde{r}_{m}^{j} }) + \tilde{w}_{l}^{s} \tilde{v}_{j}^{s}(\frac{\partial^{2} F_{s_{k}^{i}}}{         \partial \tilde{s}_{n}^{l}      \partial \tilde{s}_{m}^{j}  })] \frac{\partial}{\partial s_{k}^{i}}\\
& \hspace{10mm} + [w_{l}^{r} \tilde{v}_{j}^{s} (\frac{\partial^{2} \tilde{F}_{\tilde{r}_{k}^{i}}}{\partial r_{n}^{l} \partial \tilde{s}_{m}^{j}}) + w_{l}^{s} \tilde{v}_{j}^{r} (\frac{\partial^{2} \tilde{F}_{\tilde{r}_{k}^{i}}}{\partial s_{n}^{l} \partial \tilde{r}_{m}^{j}})\\
& \hspace{12mm} + \tilde{w}_{l}^{r} v_{j}^{s} (\frac{\partial^{2} \tilde{F}_{\tilde{r}_{k}^{i}}}{ \partial \tilde{r}_{n}^{l} \partial s_{m}^{j}   }) + \tilde{w}_{l}^{s} v_{j}^{r}(\frac{\partial^{2} \tilde{F}_{\tilde{r}_{k}^{i}}}{ \partial \tilde{s}_{n}^{l}   \partial r_{m}^{j}  })] \frac{\partial}{\partial \tilde{r}_{k}^{i}}  \\
& \hspace{10mm} + [w_{l}^{r} \tilde{v}_{j}^{r} (\frac{\partial^{2} \tilde{F}_{\tilde{s}_{k}^{i}}}{\partial r_{n}^{l} \partial \tilde{r}_{m}^{j}}) + w_{l}^{s} \tilde{v}_{j}^{s} (\frac{\partial^{2} \tilde{F}_{\tilde{s}_{k}^{i}}}{\partial s_{n}^{l} \partial \tilde{s}_{m}^{j}})\\
& \hspace{12mm} + \tilde{w}_{l}^{r} v_{j}^{r} (\frac{\partial^{2} \tilde{F}_{\tilde{s}_{k}^{i}}}{\partial \tilde{r}_{n}^{l}  \partial r_{m}^{j}  }) + \tilde{w}_{l}^{s} v_{j}^{s}(\frac{  \partial^{2} \tilde{F}_{\tilde{s}_{k}^{i}}}{    \partial \tilde{s}_{n}^{l} \partial s_{m}^{j}    })] \frac{\partial}{\partial \tilde{s}_{k}^{i}}. 
\end{split}
\end{equation} 
Furthermore, we investigate the coefficients of $\displaystyle{\frac{\partial}{\partial r_{k}^{i}}}$ from (66) as 
\begin{equation}
\begin{split} 
& \sum_{j,l=1}^{3} w_{l}^{r} v_{j}^{s} (\frac{\partial^{2} F_{r_{k}^{i}}}{\partial r_{n}^{l} \partial s_{m}^{j}}) + w_{l}^{s} v_{j}^{r} (\frac{\partial^{2} F_{r_{k}^{i}}}{\partial s_{n}^{l} \partial r_{m}^{j}}) + \tilde{w}_{l}^{r} \tilde{v}_{j}^{s} (\frac{\partial^{2} F_{r_{k}^{i}}}{\partial \tilde{r}_{n}^{l}\partial \tilde{s}_{m}^{j}}) + \tilde{w}_{l}^{s} \tilde{v}_{j}^{r}(\frac{\partial^{2} F_{r_{k}^{i}}}{\partial \tilde{s}_{n}^{l}\partial \tilde{r}_{m}^{j} })\\
=& \sum_{j,l=1}^{3} w_{l}^{r} v_{j}^{s}(\delta_{n, k-m} + \delta_{n, m-k} - \delta_{n, m+k}) \alpha_{j,l}(i,k) \\
&+ w_{l}^{s} v_{j}^{r}(\delta_{n, k-m} - \delta_{n, m-k} + \delta_{n, m+k})\alpha_{j,l}(i,k) \\
&+ \tilde{w}_{l}^{r} \tilde{v}_{j}^{s}(-\delta_{n, k-m} - \delta_{n, m-k} + \delta_{n, m+k})\alpha_{j,l}(i,k) \\
&+ \tilde{w}_{l}^{s} \tilde{v}_{j}^{r}(-\delta_{n, k-m} + \delta_{n, m-k} - \delta_{n, m+k}) \alpha_{j,l}(i,k) 
\end{split}
\end{equation} 
by (51), (50), (53), (52) where we compute 
\begin{equation}
\begin{split} 
\sum_{j,l=1}^{3} w_{l}^{r} v_{j}^{s}  \alpha_{j,l}(i,k)
=& \sum_{j,l=1}^{3} w_{l}^{r} v_{j}^{s} \left(k_{j} \delta_{i,l}  - \frac{k_{i}k_{j}k_{l}}{\lvert k\rvert^{2}}\right) + w_{l}^{r}v_{j}^{s} \left( k_{l}\delta_{i,j} - \frac{k_{i}k_{j}k_{l}}{\lvert k\rvert^{2}} \right) \\
=& (v^{s} \cdot k) P_{k}(w^{r})_{i} + (w^{r} \cdot k) P_{k}(v^{s})_{i}. 
\end{split}
\end{equation} 
Similarly, we may compute 
\begin{equation}
\sum_{j,l=1}^{3} w_{l}^{s} v_{j}^{r} \alpha_{j,l}(i,k)
= (v^{r} \cdot k) P_{k}(w^{s})_{i} + (w^{s}\cdot k) P_{k}(v^{r})_{i},  
\end{equation}
\begin{equation}
\sum_{j,l=1}^{3} \tilde{w}_{l}^{r} \tilde{v}_{j}^{s} \alpha_{j,l}(i,k)
= (\tilde{v}^{s} \cdot k) P_{k}(\tilde{w}^{r})_{i} + (\tilde{w}^{r} \cdot k) P_{k}(\tilde{v}^{s})_{i}, 
\end{equation}
\begin{equation}
\sum_{j,l=1}^{3} \tilde{w}_{l}^{s} \tilde{v}_{j}^{r} \alpha_{j,l}(i,k) 
= (\tilde{v}^{r} \cdot k)P_{k}(\tilde{w}^{s})_{i} + (\tilde{w}^{s} \cdot k) P_{k} (\tilde{v}^{r})_{i}, 
\end{equation}
so that applying (68), (69), (70), (71) in (67), we see that the coefficient of $\displaystyle{\frac{\partial}{\partial r_{k}^{i}}}$ is 
\begin{equation}
\begin{split} 
&(\delta_{n, k-m} + \delta_{n, m-k} - \delta_{n, m+k})[(v^{s} \cdot k) P_{k}(w^{r})_{i} + (w^{r}\cdot k)P_{k} (v^{s})_{i}]\\
&+ (\delta_{n, k-m} - \delta_{n, m-k} + \delta_{n, m+k}) [(v^{r} \cdot k) P_{k}(w^{s})_{i} + (w^{s} \cdot k) P_{k}(v^{r})_{i}] \\
&+ (-\delta_{n, k-m} - \delta_{n, m-k} + \delta_{n, m+k}) [(\tilde{v}^{s} \cdot k) P_{k} (\tilde{w}^{r})_{i} + (\tilde{w}^{r} \cdot k) P_{k}(\tilde{v}^{s})_{i}]\\
&+ (-\delta_{n, k-m} + \delta_{n, m-k} - \delta_{n, m+k})[(\tilde{v}^{r} \cdot k) P_{k} (\tilde{w}^{s})_{i} + (\tilde{w}^{s} \cdot k) P_{k}(\tilde{v}^{r})_{i}].
\end{split}
\end{equation} 
Next, we investigate the coefficient of $\displaystyle{\frac{\partial}{\partial s_{k}^{i}}}$ in (66) as 
\begin{equation}
\begin{split} 
& \sum_{j,l=1}^{3} w_{l}^{r} v_{j}^{r} (\frac{\partial^{2} F_{s_{k}^{i}}}{\partial r_{n}^{l} \partial r_{m}^{j}}) + w_{l}^{s} v_{j}^{s} (\frac{\partial^{2} F_{s_{k}^{i}}}{\partial s_{n}^{l} \partial s_{m}^{j}}) + \tilde{w}_{l}^{r} \tilde{v}_{j}^{r} (\frac{\partial^{2} F_{s_{k}^{i}}}{\partial \tilde{r}_{n}^{l}\partial \tilde{r}_{m}^{j}}) + \tilde{w}_{l}^{s} \tilde{v}_{j}^{s}(\frac{\partial^{2} F_{s_{k}^{i}}}{\partial \tilde{s}_{n}^{l}\partial \tilde{s}_{m}^{j} })\\
=& \sum_{j,l=1}^{3} w_{l}^{r} v_{j}^{r} (-\delta_{n, k-m} - \delta_{n, m-k} - \delta_{n, m+k})\alpha_{j,l}(i,k) \\
& \hspace{4mm} + w_{l}^{s} v_{j}^{s} (\delta_{n, k-m}  - \delta_{n, m-k} - \delta_{n, m+k}) \alpha_{j,l}(i,k) \\
&\hspace{4mm} +  \tilde{w}_{l}^{r} \tilde{v}_{j}^{r} (\delta_{n, k-m} + \delta_{n, m-k} + \delta_{n, m+k})\alpha_{j,l}(i,k)  \\
&\hspace{4mm} + \tilde{w}_{l}^{s} \tilde{v}_{j}^{s} (-\delta_{n, k-m} + \delta_{n, m-k} + \delta_{n, m+k})\alpha_{j,l}(i,k),  
\end{split}
\end{equation} 
by (54), (55), (56), (57) where 
\begin{equation}
 \sum_{j,l=1}^{3} w_{l}^{r} v_{j}^{r} \alpha_{j,l}(i,k)
= (v^{r} \cdot k) P_{k}(w^{r})_{i} + (w^{r} \cdot k) P_{k}(v^{r})_{i}, 
\end{equation}
\begin{equation}
 \sum_{j,l=1}^{3} w_{l}^{s} v_{j}^{s} \alpha_{j,l}(i,k) 
= (v^{s} \cdot k) P_{k}(w^{s})_{i} + (w^{s} \cdot k) P_{k}(v^{s})_{i},
\end{equation}
\begin{equation}
 \sum_{j,l=1}^{3} \tilde{w}_{l}^{r} \tilde{v}_{j}^{r} \alpha_{j,l}(i,k) 
= (\tilde{v}^{r} \cdot k) P_{k} (\tilde{w}^{r})_{i} + (\tilde{w}^{r} \cdot k) P_{k}(\tilde{v}^{r})_{i},
\end{equation}
\begin{equation}
\sum_{j,l=1}^{3} \tilde{w}_{l}^{s} \tilde{v}_{j}^{s} \alpha_{j,l}(i,k)
= (\tilde{v}^{s} \cdot k) P_{k}(\tilde{w}^{s})_{i} + (\tilde{w}^{s} \cdot k) P_{k} (\tilde{v}^{s})_{i},
\end{equation}
so that applying (74), (75), (76), (77) in (73), we see that the coefficient of $\displaystyle{\frac{\partial}{\partial s_{k}^{i}}}$ is 
\begin{equation}
\begin{split} 
& \hspace{3mm} (-\delta_{n, k-m} - \delta_{n, m-k} - \delta_{n, m+k}) [(v^{r} \cdot k) P_{k}(w^{r})_{i} + (w^{r} \cdot k) P_{k}(v^{r})_{i}] \\
&+ (\delta_{n, k-m} - \delta_{n, m-k} - \delta_{n, m+k}) [(v^{s} \cdot k) P_{k}(w^{s})_{i} + (w^{s} \cdot k)P_{k}(v^{s})_{i}]\\
&+ (\delta_{n, k-m} + \delta_{n, m-k} + \delta_{n, m+k}) [(\tilde{v}^{r} \cdot k) P_{k}(\tilde{w}^{r})_{i} + (\tilde{w}^{r} \cdot k) P_{k}(\tilde{v}^{r})_{i}] \\
&+ (-\delta_{n, k-m} + \delta_{n, m-k} + \delta_{n, m+k})[(\tilde{v}^{s} \cdot k) P_{k}(\tilde{w}^{s})_{i} + (\tilde{w}^{s} \cdot k) P_{k}(\tilde{v}^{s})_{i}]. 
\end{split}
\end{equation} 
Next, we investigate the coefficient of $\displaystyle{\frac{\partial}{\partial \tilde{r}_{k}^{i}}}$ in (66) as 
\begin{equation}
\begin{split} 
& \sum_{j,l=1}^{3} w_{l}^{r} \tilde{v}_{j}^{s} (\frac{\partial^{2} \tilde{F}_{\tilde{r}_{k}^{i}}}{\partial r_{n}^{l} \partial \tilde{s}_{m}^{j}}) + w_{l}^{s} \tilde{v}_{j}^{r} (\frac{\partial^{2} \tilde{F}_{\tilde{r}_{k}^{i}}}{\partial s_{n}^{l} \partial \tilde{r}_{m}^{j}}) + \tilde{w}_{l}^{r} v_{j}^{s} (\frac{\partial^{2} \tilde{F}_{\tilde{r}_{k}^{i}}}{\partial \tilde{r}_{n}^{l}\partial s_{m}^{j}}) + \tilde{w}_{l}^{s} v_{j}^{r}(\frac{\partial^{2} \tilde{F}_{\tilde{r}_{k}^{i}}}{\partial \tilde{s}_{n}^{l}\partial r_{m}^{j} })\\
& \sum_{j,l=1}^{3} w_{l}^{r} \tilde{v}_{j}^{s} (\delta_{n, k-m}  + \delta_{n, m-k}- \delta_{n, k+m} )(k_{l}\delta_{i,j}  - k_{j}\delta_{i,l} )\\
&+ w_{l}^{s}\tilde{v}_{j}^{r} (\delta_{n, k-m} - \delta_{n, m-k} + \delta_{n, k+m})(k_{l}\delta_{i,j}  - k_{j}\delta_{i,l} )\\
&+ \tilde{w}_{l}^{r}v_{j}^{s} (\delta_{n, k-m} + \delta_{n, m-k} - \delta_{n, m+k})(-k_{l}\delta_{i,j}  + k_{j}\delta_{i,l} )\\
&+ \tilde{w}_{l}^{s}v_{j}^{r}(\delta_{n, k-m} - \delta_{n, m-k} + \delta_{n, m+k})(-k_{l}\delta_{i,j}  + k_{j}\delta_{i,l} )\\
=& (\delta_{n, k-m} + \delta_{n, m-k}- \delta_{n, k+m} ) (\tilde{v}_{i}^{s} (w^{r} \cdot k) - w_{i}^{r} (\tilde{v}^{s} \cdot k))\\
&+ (\delta_{n, k-m} - \delta_{n, m-k}+ \delta_{n, k+m} ) (\tilde{v}_{i}^{r} (w^{s} \cdot k) - w_{i}^{s} (\tilde{v}^{r} \cdot k))\\
&+ (\delta_{n, k-m} + \delta_{n,m-k} - \delta_{n, m+k}) (-v_{i}^{s}(\tilde{w}^{r} \cdot k) + \tilde{w}_{i}^{r} (v^{s} \cdot k))\\
&+ (\delta_{n, k-m} - \delta_{n, m-k} + \delta_{n, m+k})(-v_{i}^{r} (\tilde{w}^{s}\cdot k) + \tilde{w}_{i}^{s} (v^{r} \cdot k))
\end{split}
\end{equation} 
by (59), (61), (60), (58). Similarly, we may rewrite the coefficient of $\displaystyle{\frac{\partial}{\partial \tilde{s}_{k}^{i}}}$ in (66): 
\begin{equation}
\begin{split} 
& \sum_{j,l=1}^{3} w_{l}^{r} \tilde{v}_{j}^{r} (\frac{\partial^{2} \tilde{F}_{\tilde{s}_{k}^{i}}}{\partial r_{n}^{l} \partial \tilde{r}_{m}^{j}}) + w_{l}^{s} \tilde{v}_{j}^{s} (\frac{\partial^{2} \tilde{F}_{\tilde{s}_{k}^{i}}}{\partial s_{n}^{l} \partial \tilde{s}_{m}^{j}}) + \tilde{w}_{l}^{r} v_{j}^{r} (\frac{\partial^{2} \tilde{F}_{\tilde{s}_{k}^{i}}}{\partial \tilde{r}_{n}^{l}\partial r_{m}^{j}}) + \tilde{w}_{l}^{s} v_{j}^{s}(\frac{\partial^{2} \tilde{F}_{\tilde{s}_{k}^{i}}}{\partial \tilde{s}_{n}^{l}\partial s_{m}^{j} })\\
=& \sum_{j, l=1}^{3} w_{l}^{r} \tilde{v}_{j}^{r} (\delta_{n, k-m}+ \delta_{n, m-k} + \delta_{n, k+m} ) (-k_{l}\delta_{i,j}  + k_{j}\delta_{i,l} ) \\
&+ w_{l}^{s} \tilde{v}_{j}^{s}(\delta_{n, k-m}- \delta_{n, m-k} - \delta_{n, k+m} ) (k_{l}\delta_{i,j}  - k_{j}\delta_{i,l} ) \\
&+ \tilde{w}_{l}^{r} v_{j}^{r} (-\delta_{n, k-m} - \delta_{n, m-k} - \delta_{n, m+k}) (-k_{l}\delta_{i,j} + k_{j}\delta_{i,l} )\\
&+ \tilde{w}_{l}^{s} v_{j}^{s} (\delta_{n, k-m} - \delta_{n, m-k} - \delta_{n, m+k}) (-k_{l}\delta_{i,j}  + k_{j}\delta_{i,l} )\\
=& (\delta_{n, k-m}+ \delta_{n, m-k} + \delta_{n, k+m} ) (-\tilde{v}_{i}^{r}(w^{r} \cdot k) + w_{i}^{r} (\tilde{v}^{r} \cdot k)) \\
&+ (\delta_{n, k-m}- \delta_{n, m-k} - \delta_{n, k+m}) (\tilde{v}_{i}^{s}(w^{s} \cdot k) - w_{i}^{s} (\tilde{v}^{s} \cdot k))\\
&+ (-\delta_{n, k-m} - \delta_{n, m-k} - \delta_{n, m+k})(-v_{i}^{r}(\tilde{w}^{r} \cdot k) + \tilde{w}_{i}^{r}(v^{r} \cdot k))\\
&+ (\delta_{n, k-m} - \delta_{n, m-k} - \delta_{n, m+k}) (-v_{i}^{s}(\tilde{w}^{s} \cdot k) + \tilde{w}_{i}^{s}(v^{s} \cdot k)) 
\end{split}
\end{equation} 
by (63), (65), (62), (64). Therefore, we conclude applying (72), (78), (79), (80) in (66), 
\begin{equation*}
\begin{split} 
& [[F_{0}, V], W]\\
=& \sum_{k \in \tilde{\mathcal{K}}} [(\delta_{n, k-m} + \delta_{n, m-k} - \delta_{n, m+k})[(v^{s} \cdot k) P_{k}(w^{r}) + (w^{r}\cdot k)P_{k} (v^{s})]\nonumber\\
&+ (\delta_{n, k-m} - \delta_{n, m-k} + \delta_{n, m+k}) [(v^{r} \cdot k) P_{k}(w^{s}) + (w^{s} \cdot k) P_{k}(v^{r})] \nonumber\\
&+ (-\delta_{n, k-m} - \delta_{n, m-k} + \delta_{n, m+k}) [(\tilde{v}^{s} \cdot k) P_{k} (\tilde{w}^{r}) + (\tilde{w}^{r} \cdot k) P_{k}(\tilde{v}^{s})]\nonumber\\
&+ (-\delta_{n, k-m} + \delta_{n, m-k} - \delta_{n, m+k})[(\tilde{v}^{r} \cdot k) P_{k} (\tilde{w}^{s}) + (\tilde{w}^{s} \cdot k) P_{k}(\tilde{v}^{r})]] \cdot \frac{\partial}{\partial r_{k}}\nonumber\\
&+ [(-\delta_{n, k-m} - \delta_{n, m-k} - \delta_{n, m+k}) [(v^{r} \cdot k) P_{k}(w^{r}) + (w^{r} \cdot k) P_{k}(v^{r})] \nonumber\\
&+ (\delta_{n, k-m} - \delta_{n, m-k} - \delta_{n, m+k}) [(v^{s} \cdot k) P_{k}(w^{s}) + (w^{s} \cdot k)P_{k}(v^{s})]\nonumber\\
&+ (\delta_{n, k-m} + \delta_{n, m-k} + \delta_{n, m+k}) [(\tilde{v}^{r} \cdot k) P_{k}(\tilde{w}^{r}) + (\tilde{w}^{r} \cdot k) P_{k}(\tilde{v}^{r})]\nonumber\\
&+ (-\delta_{n, k-m} + \delta_{n, m-k} + \delta_{n, m+k})[(\tilde{v}^{s} \cdot k) P_{k}(\tilde{w}^{s}) + (\tilde{w}^{s} \cdot k) P_{k}(\tilde{v}^{s})]] \cdot \frac{\partial}{\partial s_{k}}\nonumber\\
&+ [(\delta_{n, k-m}+ \delta_{n, m-k} - \delta_{n, k+m} ) (\tilde{v}^{s} (w^{r} \cdot k) - w^{r} (\tilde{v}^{s} \cdot k))\nonumber\\
&+ (\delta_{n, k-m}- \delta_{n, m-k} + \delta_{n, k+m} ) (\tilde{v}^{r} (w^{s} \cdot k) - w^{s} (\tilde{v}^{r} \cdot k))\nonumber\\
&+ (\delta_{n, k-m} + \delta_{n,m-k} - \delta_{n, m+k}) (-v^{s}(\tilde{w}^{r} \cdot k) + \tilde{w}^{r} (v^{s} \cdot k))\nonumber\\
&+ (\delta_{n, k-m} - \delta_{n, m-k} + \delta_{n, m+k})(-v^{r} (\tilde{w}^{s}\cdot k) + \tilde{w}^{s} (v^{r} \cdot k))] \cdot \frac{\partial}{\partial \tilde{r}_{k}}\nonumber\\
&+ [(\delta_{n, k-m} + \delta_{n, m-k}+ \delta_{n, k+m} ) (-\tilde{v}^{r}(w^{r} \cdot k) + w^{r} (\tilde{v}^{r} \cdot k)) \nonumber\\
&+ (\delta_{n, k-m}- \delta_{n, m-k} - \delta_{n, k+m}  ) (\tilde{v}^{s}(w^{s} \cdot k) - w^{s} (\tilde{v}^{s} \cdot k))\nonumber\\
&+ (-\delta_{n, k-m} - \delta_{n, m-k} - \delta_{n, m+k})(-v^{r}(\tilde{w}^{r} \cdot k) + \tilde{w}^{r}(v^{r} \cdot k))\nonumber\\
&+ (\delta_{n, k-m} - \delta_{n, m-k} - \delta_{n, m+k}) (-v^{s}(\tilde{w}^{s} \cdot k) + \tilde{w}^{s}(v^{s} \cdot k))  ] \cdot \frac{\partial}{\partial \tilde{s}_{k}}.\nonumber
\end{split}
\end{equation*} 
Now recalling that $k = m + n, h = n-m, g = m-n$ by hypothesis, we deduce (33) to conclude the  proof of Proposition 4.1. 
\end{proof}

\begin{proposition}
Suppose $\mathcal{N}$ is a subset of indices and   
\small 
\begin{equation*}
\begin{split} 
A(\mathcal{N}) \triangleq \{k \in \mathcal{K}_{N}:& \text{ constant vector fields corresponding to } k \text{ if } k \in \tilde{\mathcal{K}}, -k \text{ if } k \in - \tilde{\mathcal{K}},\\
&  \text{ are in the Lie algebra generated by the vector fields } \{\mathcal{F}_{0}\} \cup \mathcal{U}_{j}, j \in \mathcal{N}\}. 
\end{split} 
\end{equation*}
\normalsize 

\begin{enumerate}
\item If $m \in A(\mathcal{N})$, then $-m \in A (\mathcal{N})$. 
\item If $m, n \in A(\mathcal{N}), m + n \in \mathcal{K}_{N}, m$ and $n$ are linearly independent and $\lvert m \rvert \neq \lvert n \rvert$, then $m + n \in A(\mathcal{N})$. 
\item If $\lvert m \rvert = \lvert n \rvert, v \in \mathcal{U}_{m}, w \in \mathcal{U}_{n}$, then $[[F_{0}, V], W]$ spans the four-dimensional subspace of $\mathcal{U}_{m+n}$. 
\end{enumerate}

\end{proposition}

\begin{proof}

Having derived (33) through detailed computations, the proof of this Proposition 4.2 goes through very similarly to the previous work on the NSE in \cite{R04} and Boussinesq equations in \cite{LW04}; we sketch the proof for completeness. 

For the first part, by (3) $u_{-k} = \overline{u}_{k}, b_{-k} = \overline{b}_{k}$. Thus, if $m \in A(\mathcal{N})$, then it follows that $-m \in A(\mathcal{N})$. 

For the second part, we take $m, n \in A(\mathcal{N})$ such that $m + n \in \mathcal{K}_{N}, m, n$ are linearly independent and $\lvert m \rvert \neq \lvert n \rvert$. By (20) $\mathcal{K}_{N} = \tilde{\mathcal{K}} \cup (-\tilde{\mathcal{K}})$ where $\tilde{\mathcal{K}} \cap (-\tilde{\mathcal{K}}) = \emptyset$. Thus, it suffices to show that if $m, n \in A(\mathcal{N}) \cap \tilde{\mathcal{K}}$ such that $k = m + n \in \tilde{\mathcal{K}}$, then $m + n \in A(\mathcal{N})$ because we can repeat the proof in case $k \in - \tilde{\mathcal{K}}$ identically. Now we let 
\begin{equation*}
\begin{split} 
& V^{r} \triangleq \sum_{j=1}^{3} v_{j} \frac{\partial}{\partial r_{m}^{j}} + \tilde{v}_{j} \frac{\partial}{\partial \tilde{r}_{m}^{j}}, \hspace{5mm} V^{s} \triangleq \sum_{j=1}^{3} v_{j} \frac{\partial}{\partial s_{m}^{j}} + \tilde{v}_{j} \frac{\partial}{\partial \tilde{s}_{m}^{j}}\\
& W^{r} \triangleq \sum_{l=1}^{3} w_{l} \frac{\partial}{\partial r_{n}^{l}} + \tilde{w}_{l} \frac{\partial}{\partial \tilde{r}_{n}^{l}}, \hspace{5mm}
 W^{s} \triangleq \sum_{l=1}^{3} w_{l} \frac{\partial}{\partial s_{n}^{l}} + \tilde{w}_{l} \frac{\partial}{\partial \tilde{s}_{n}^{l}},  
\end{split} 
\end{equation*}
where $v \cdot m = w \cdot n = \tilde{v} \cdot m = \tilde{w} \cdot n = 0$ and compute
\begin{equation*}
\begin{split} 
& [[F_{0}, V^{r}], W^{s}] + [[F_{0}, V^{s}], W^{r}]\\
=& 2[(v\cdot k) P_{k}(w) + (w\cdot k) P_{k}(v) - (\tilde{v}\cdot k) P_{k}(\tilde{w}) - (\tilde{w}\cdot k) P_{k}(\tilde{v})] \cdot \frac{\partial}{\partial r_{k}}\nonumber\\
&+ 2[-(\tilde{v} \cdot k) w + (w\cdot k) \tilde{v} + (v\cdot k) \tilde{w} - (\tilde{w} \cdot k) v ] \cdot \frac{\partial}{\partial \tilde{r}_{k}}, 
\end{split} 
\end{equation*}
\begin{equation*}
\begin{split} 
& [[F_{0}, V^{r}], W^{r}] - [[F_{0}, V^{s}], W^{s}]\\
=& -2[(v\cdot k) P_{k}(w) + (w\cdot k) P_{k}(v) - (\tilde{v}\cdot k) P_{k}(\tilde{w}) - (\tilde{w} \cdot k) P_{k}(\tilde{v})]\cdot \frac{\partial}{\partial s_{k}}\nonumber\\
&-2 [-(\tilde{v}\cdot k) w + (w\cdot k) \tilde{v} +(v\cdot k) \tilde{w} - (\tilde{w}\cdot k) v ] \cdot \frac{\partial}{\partial \tilde{s}_{k}} 
\end{split} 
\end{equation*}
due to (33). We take $H, L \in \mathbb{R}^{3}$ so that $\{k, H, L\}$ is a basis of $\mathbb{R}^{3}$, $H, L$ span $\{x \in \mathbb{R}^{3}: x\cdot k = 0 \}$ and $m, n \in \text{span} [k, H]$. By assumption, $m$ and $n$ are linearly independent and $\lvert m \rvert \neq \lvert n \rvert$ and hence if 
\begin{equation*}
v = \alpha_{1} k + \beta_{1}H + \gamma_{1}L, \hspace{5mm} w = \alpha_{2} k + \beta_{2}H + \gamma_{2} L, 
\end{equation*}
then, it is possible to choose the coefficients $\alpha_{1}, \alpha_{2}, \beta_{1}, \beta_{2}, \gamma_{1}, \gamma_{2}$ so that $(v\cdot k) P_{k} (w) + (w\cdot k) P_{k}(v)$ can be any vector in span $[H, L ]$. Since we can freely choose any $\tilde{v}, \tilde{w}$, we may have 
\begin{equation*}
-(\tilde{v}\cdot k) P_{k} (\tilde{w}) - (\tilde{w}\cdot k) P_{k}(\tilde{v}), \hspace{3mm} -(\tilde{v}\cdot k) w + (w\cdot k) \tilde{v} + (v\cdot k) \tilde{w} - (\tilde{w}\cdot k) v 
\end{equation*}
to be any vector. Therefore, $\mathcal{U}_{k}$ is contained in the Lie algebra generated by the vector field $\{F_{0}\} \cup \mathcal{U}_{p}, p \in \mathcal{N}$, which implies by definition of $A(\mathcal{N})$ that $k \in A(\mathcal{N})$ and hence $m + n \in A(\mathcal{N})$. 

The third part can be proven similarly to the proof of the second part (see \cite[Lemma 4.2 (iii)]{LW04}, proof of \cite[Proposition 5.2]{R04}). This completes the proof of Proposition 4.2. 
\end{proof}

\begin{proposition}
If $\mathcal{N}$ contains $(1,0,0), (0,1,0)$ and $(0,0,1)$, then $A(\mathcal{N}) = \mathcal{K}_{N}$ so that the transition probability densities of the solution process to (\ref{15}), (\ref{16}) are regular. 
\end{proposition}

\begin{proof}
From Proposition 4.2 (3), summing $(1,0,0), (0,1,0)$ gives four-dimensional subspace of $\mathcal{U}_{(1,1,0)}$, which combining with $\mathcal{U}_{(0,0,1)}$ gives $\mathcal{U}_{(1,1,1)}$. Subtracting $(0,0,1)$ from $(1,1,1)$ also gives $\mathcal{U}_{(1,1,0)}$. Similarly we can obtain all indices of norm two as well as all indices in $\mathcal{K}_{N}$, for any $N$ that is \emph{a priori} fixed. Therefore, $\mathcal{K}_{N} \subset A(\mathcal{N})$ and thus $A(\mathcal{N}) = \mathcal{K}_{N}$ (by definition of $A(\mathcal{N})$, we have $A(\mathcal{N}) \subset \mathcal{K}_{N}$). Hence, by H$\ddot{\mathrm{o}}$rmander's condition of hypoellipticity, transition semigroup generated by (\ref{15}), (\ref{16}) is regular (see e.g. \cite{S83}). This completes the proof of Proposition 4.3. 
\end{proof}

\subsection{Recurrence of neighborhoods of the origin}

In this section, following \cite[Section 3]{EM01}, except that therein the domain was $\mathbb{T}^{2}$ which makes it easier than the current case of $\mathbb{T}^{3}$, we show that in system (\ref{15}), (\ref{16}), starting anywhere, fixing any neighborhood of the origin, the corresponding dynamics enters this neighborhood infinitely often. We denote by 
\begin{equation*}
\mathcal{B}(c) \triangleq \left\{X \in L^{2}(\mathbb{T}^{3}): \lVert X \rVert_{L^{2}} = \left(\int_{\mathbb{T}^{3}} \lvert X(x) \rvert^{2} dx \right)^{\frac{1}{2}} \leq c \right\}. 
\end{equation*}
\begin{proposition}
For fixed constants $C_{0}, C_{1} > 0$, suppose that $\mathcal{B}_{0} \triangleq \mathcal{B}(C_{0}), \mathcal{B}_{1} \triangleq \mathcal{B}(C_{1})$ are two arbitrary balls around the origin; additionally suppose $h > 0$ is also fixed. Then there exists $T_{0} = T_{0} (C_{0}, C_{1}) > 0$ such that for any $T \geq T_{0}$, there exists a constant $p^{\ast} > 0$ such that 
\begin{equation*}
\inf_{(u_{0}, b_{0}) \in \mathcal{B}_{0}} \mathbb{P}_{(u_{0}, b_{0})} \{ (u(t), b(t)) \in \mathcal{B}_{1} \text{ for all } t \in [T, T+h]\} \geq p^{\ast}.
\end{equation*}
\end{proposition}

\begin{proof}
We define $P_{N}$ to be the projection operator onto the Fourier modes less than or equal to $N$ in absolute value, multiply (\ref{15}), (\ref{16}) by $e^{ik\cdot x}$ and then sum over $k \in \mathcal{K}_{N}$ to deduce  
\begin{subequations}
\begin{align}
& du = [\Delta u - P_{N} \mathcal{P} ((u\cdot \nabla) u) + P_{N} \mathcal{P} ((b\cdot \nabla) b) ]dt + dW_{u},\\
& db = [\Delta b - P_{N} ((u\cdot \nabla) b) + P_{N} ((b\cdot\nabla) u) ] dt + dW_{b}. 
\end{align}
\end{subequations}
We define $v(t) \triangleq u(t) - \tilde{f}_{u}(t), \tilde{f}_{u}(t) \triangleq W_{u}(t) - W_{u}(0), B(t) \triangleq b(t) - \tilde{f}_{b}(t), \tilde{f}_{b}(t) \triangleq W_{b}(t) - W_{b}(0)$ so that 
\begin{equation}
\partial_{t} v= \Delta v - P_{N} \mathcal{P} ((u\cdot\nabla) u) + P_{N} \mathcal{P} ((b\cdot\nabla) b) + \Delta \tilde{f}_{u}, 
\end{equation}
\begin{equation}
\partial_{t}B = \Delta B - P_{N} ((u\cdot\nabla) b) + P_{N} ((b\cdot\nabla) u) + \Delta \tilde{f}_{b},
\end{equation}
by (81a), (81b). Taking $L^{2}$-inner products of (82), (83) with $v, B$ respectively give 
\begin{equation}
\begin{split} 
& \frac{1}{2}\partial_{t} (\lVert v \rVert_{L^{2}}^{2} + \lVert B \rVert_{L^{2}}^{2}) + \lVert \nabla v \rVert_{L^{2}}^{2} + \lVert \nabla B \rVert_{L^{2}}^{2} \\
=& -\int (u\cdot\nabla) \tilde{f}_{u} \cdot v  - \int (u\cdot\nabla) \tilde{f}_{b} \cdot B \\
& + \int (b\cdot\nabla) \tilde{f}_{b} \cdot v + (b\cdot\nabla) \tilde{f}_{u} \cdot B + \int \Delta \tilde{f}_{u} \cdot v + \Delta \tilde{f}_{b} \cdot B 
\end{split} 
\end{equation}
where we used (\ref{1c}), as well as the noise and consequently those of $v, B$. We estimate 
\begin{equation}
\begin{split} 
& -\int (u\cdot\nabla) \tilde{f}_{u} \cdot v - \int (u\cdot\nabla) \tilde{f}_{b} \cdot B + \int (b\cdot\nabla) \tilde{f}_{b} \cdot v + \int (b\cdot\nabla) \tilde{f}_{u} \cdot B \\
\lesssim& (\lVert u \rVert_{L^{4}} + \lVert b \rVert_{L^{4}}) (\lVert \nabla \tilde{f}_{u} \rVert_{L^{4}} + \lVert \nabla \tilde{f}_{b} \rVert_{L^{4}})(\lVert v \rVert_{L^{2}} + \lVert B \rVert_{L^{2}}) \\
\lesssim& (\lVert \nabla u\rVert_{L^{2}} + \lVert \nabla b \rVert_{L^{2}}) (\lVert \Delta \tilde{f}_{u} \rVert_{L^{2}} + \lVert \Delta \tilde{f}_{b} \rVert_{L^{2}})(\lVert v \rVert_{L^{2}} + \lVert B \rVert_{L^{2}})  \\
\lesssim& (\lVert \nabla v\rVert_{L^{2}} + \lVert \nabla B \rVert_{L^{2}})(\lVert \Delta \tilde{f}_{u} \rVert_{L^{2}} + \lVert \Delta \tilde{f}_{b} \rVert_{L^{2}})(\lVert v \rVert_{L^{2}} + \lVert B \rVert_{L^{2}}) \\
& + (\lVert \nabla \tilde{f}_{u} \rVert_{L^{2}} + \lVert \nabla \tilde{f}_{b} \rVert_{L^{2}})(\lVert \Delta \tilde{f}_{u} \rVert_{L^{2}} + \lVert \Delta \tilde{f}_{b} \rVert_{L^{2}}) (\lVert \nabla v\rVert_{L^{2}} + \lVert \nabla B \rVert_{L^{2}}) \\
\leq& \frac{1}{8} (\lVert \nabla v \rVert_{L^{2}}^{2} + \lVert \nabla B \rVert_{L^{2}}^{2}) + C(\lVert \Delta \tilde{f}_{u} \rVert_{L^{2}}^{2} + \lVert \Delta \tilde{f}_{b} \rVert_{L^{2}}^{2})(\lVert v \rVert_{L^{2}}^{2} + \lVert B \rVert_{L^{2}}^{2})\\
& \hspace{33mm} + C(\lVert \nabla \tilde{f}_{u} \rVert_{L^{2}}^{2} + \lVert \nabla \tilde{f}_{b} \rVert_{L^{2}}^{2} ) (\lVert \Delta \tilde{f}_{u} \rVert_{L^{2}}^{2} + \lVert \Delta \tilde{f}_{b} \rVert_{L^{2}}^{2}) 
\end{split} 
\end{equation}
where we used H$\ddot{\mathrm{o}}$lder's inequality, Sobolev embedding $H^{1}(\mathbb{T}^{3}) \hookrightarrow L^{4}(\mathbb{T}^{3})$, Poincar$\acute{\mathrm{e}}$ and Young's inequalities. Moreover, we have the estimate of 
\begin{equation*}
\begin{split} 
\int \Delta \tilde{f}_{u} \cdot v + \Delta \tilde{f}_{b} \cdot B 
\leq& \lVert \nabla \tilde{f}_{u} \rVert_{L^{2}} \lVert \nabla v \rVert_{L^{2}} + \lVert \nabla \tilde{f}_{b} \rVert_{L^{2}} \lVert \nabla B \rVert_{L^{2}}\\
\leq& \frac{1}{8} (\lVert \nabla v \rVert_{L^{2}}^{2} + \lVert \nabla B \rVert_{L^{2}}^{2}) + C ( \lVert \Delta \tilde{f}_{u} \rVert_{L^{2}}^{2} + \lVert \Delta \tilde{f}_{b} \rVert_{L^{2}}^{2}),
\end{split} 
\end{equation*}
by H$\ddot{\mathrm{o}}$lder's, Young's and Poincar$\acute{\mathrm{e}}$ inequalities; this computation along with (85) applied to (84) give
\begin{equation}
\begin{split} 
& \frac{1}{2} \partial_{t} (\lVert v\rVert_{L^{2}}^{2} + \lVert B \rVert_{L^{2}}^{2}) \\
\leq& -\frac{3}{4} \left(\lVert \nabla v \rVert_{L^{2}}^{2} + \lVert \nabla B \rVert_{L^{2}}^{2}\right) + C_{2}(\lVert \Delta \tilde{f}_{u} \rVert_{L^{2}}^{2} + \lVert \Delta \tilde{f}_{b} \rVert_{L^{2}}^{2})(\lVert v \rVert_{L^{2}}^{2} + \lVert B \rVert_{L^{2}}^{2}) \\
& \hspace{10mm}  + C (\lVert \Delta \tilde{f}_{u} \rVert_{L^{2}}^{2} + \lVert \Delta \tilde{f}_{b} \rVert_{L^{2}}^{2} + \lVert \Delta \tilde{f}_{u} \rVert_{L^{2}}^{4} + \lVert \Delta \tilde{f}_{b} \rVert_{L^{2}}^{4}) \\
\leq& -\left(\frac{3}{4} - C_{2} (\lVert \Delta \tilde{f}_{u} \rVert_{L^{2}}^{2} + \lVert \Delta \tilde{f}_{b} \rVert_{L^{2}}^{2})\right) (\lVert v \rVert_{L^{2}}^{2} + \lVert B \rVert_{L^{2}}^{2}) \\
&+ C (\lVert \Delta \tilde{f}_{u} \rVert_{L^{2}}^{2} + \lVert \Delta \tilde{f}_{b} \rVert_{L^{2}}^{2} + \lVert \Delta \tilde{f}_{u} \rVert_{L^{2}}^{4} + \lVert \Delta \tilde{f}_{b} \rVert_{L^{2}}^{4}) 
\end{split}
\end{equation}
for some $C, C_{2} > 0$. Now we fix any $\delta > 0$ and define for any $T > 0$, 
\begin{equation*}
\Omega' (\delta, T) \triangleq \left\{g \in C([0, T+h]: L^{2}(\mathbb{T}^{3})): \sup_{t \in [0,T]} \lVert \Delta g(t) \rVert_{L^{2}}^{2} \leq \min \{\delta, \frac{3}{16C_{2}} \}\right\}.
\end{equation*}
Thus, if $\tilde{f}_{u}, \tilde{f}_{b} \in \Omega' (\delta, T)$, then 
\begin{equation*}
C_{2} (\lVert \Delta \tilde{f}_{u} \rVert_{L^{2}}^{2} + \lVert \Delta \tilde{f}_{b} \rVert_{L^{2}}^{2}) \leq \frac{3}{8}
\end{equation*}
so that (86) leads to 
\begin{equation}
\begin{split} 
&\lVert v(t) \rVert_{L^{2}}^{2} + \lVert B(t) \rVert_{L^{2}}^{2}\\
\leq& (\lVert v(0) \rVert_{L^{2}}^{2} + \lVert B(0) \rVert_{L^{2}}^{2}) e^{-\frac{3}{4} t} + C \left( (\min\{\delta, \frac{3}{16 C_{2 }} \}) + (\min\{\delta, \frac{3}{16 C_{2 }} \})^{2}\right). 
\end{split}
\end{equation} 
Therefore, if $\lVert u(0) \rVert_{L^{2}}^{2} + \lVert b(0) \rVert_{L^{2}}^{2} < C_{0}$, then for any $C_{1} > 0$, we can take $T> 0$ large enough so that $\displaystyle{e^{-\frac{3}{4} T} < \frac{C_{1}}{16 C_{0}}}$, and then $\delta > 0$ small enough so that 
\begin{equation*}
C \left( (\min\{\delta, \frac{3}{16 C_{2 }} \}) + (\min\{\delta, \frac{3}{16 C_{2 }} \})^{2}\right) < \frac{C_{1}}{16}
\end{equation*}
which leads to 
\begin{equation*}
\displaystyle{(\lVert v\rVert_{L^{2}}^{2} + \lVert B \rVert_{L^{2}}^{2})(T) < \frac{C_{1}}{8}}
\end{equation*}
due to (87). This implies that because $\displaystyle{\lVert v(t) \rVert_{L^{2}}^{2} + \lVert B(t) \rVert_{L^{2}}^{2}}$ is decreasing, it holds that $\displaystyle{\lVert v(t) \rVert_{L^{2}}^{2} + \lVert B(t) \rVert_{L^{2}}^{2} < \frac{C_{1}}{8}}$ for all $t \in [T, T+h]$.

Now because $\tilde{f}_{u}, \tilde{f}_{b} \in \Omega' (\delta, T)$, by taking $\delta$ small enough, we may assume $\displaystyle{\sup_{t \in [T, T+h]} \left(\lVert (\tilde{f}_{u}, \tilde{f}_{b})(t) \rVert_{L^{2}}^{2}\right) < \frac{C_{1}}{8}}$. Therefore, for any $t \in [T, T+h]$, because $(a+b)^{p} \leq 2^{p} (a+b)^{p}$ for any $a, b \geq 0, p \in [0, \infty)$, 
\begin{equation*}
\lVert u(t) \rVert_{L^{2}}^{2} + \lVert b(t) \rVert_{L^{2}}^{2} \leq C_{1}. 
\end{equation*}
Finally, for any $T \in (0, \infty), \delta > 0, \Omega'(\delta, T)$ is an open set in the supremum topology and hence $\mathbb{P} (\Omega' (\delta, T)) > 0$. This completes the proof of Proposition 4.4. 
\end{proof}

\begin{proposition}
Suppose $\lVert u(0) \rVert_{L^{2}}^{2} + \lVert b(0) \rVert_{L^{2}}^{2} > C^{2}, \mathbb{P}$-a.s. for some $C> 0$ such that 
\begin{equation}
C^{2} > \epsilon_{0}^{u} + \epsilon_{0}^{b}  \text{ where } \epsilon_{0}^{u} \triangleq \sum_{\lvert k \rvert \leq N} \lvert q_{uk}\rvert^{2}, \hspace{1mm} \epsilon_{0}^{b} \triangleq \sum_{\lvert k \rvert \leq N} \lvert q_{bk} \rvert^{2} \text{ and } \delta \triangleq 1 - \frac{(\epsilon_{0}^{u} + \epsilon_{0}^{b})}{C^{2}}.
\end{equation}
Then 
\begin{equation}
\mathbb{P} \{\tau_{\mathcal{C}} (u(0), b(0)) \geq t \} \leq \frac{\mathbb{E} [\lVert u(0) \rVert_{L^{2}}^{2} + \lVert b(0) \rVert_{L^{2}}^{2}]}{C^{2}} e^{-2\delta t} 
\end{equation}
where $\mathcal{C} \triangleq \mathcal{B}(C)$ and 
\begin{equation*}
\tau_{\mathcal{C}} \triangleq   \inf\{t > 0: (u(t), b(t)) \in \mathcal{C} \text{ given } (u(0), b(0)) \}.
\end{equation*}
\end{proposition}

\begin{proof}
We define 
\begin{equation*}
Y(t,u) \triangleq  e^{2\delta t} \lVert u(t) \rVert_{L^{2}}^{2}, \hspace{3mm} Z(t,b) \triangleq e^{2\delta t} \lVert b(t) \rVert_{L^{2}}^{2}
\end{equation*}
and apply Ito's formula with $F(x,t) = x^{2}$ and subsequently with $F(x,t) = e^{2\delta t} x$ on (81a), (81b) to deduce in sum 
\begin{equation}
\begin{split} 
&d(Y+Z)(t)\\
=& 2\delta e^{2\delta t} (\lVert u\rVert_{L^{2}}^{2} + \lVert b\rVert_{L^{2}}^{2}) dt + e^{2\delta t} [-2 \lVert \nabla u\rVert_{L^{2}}^{2} - 2 \lVert \nabla b\rVert_{L^{2}}^{2} + \epsilon_{0}^{u} + \epsilon_{0}^{b}]dt  \\
&+ e^{2\delta t} [(2u, dW_{u}) + (2b, dW_{b})]  \\
\leq& \left( \epsilon_{0}^{u} + \epsilon_{0}^{b} - 2 \left(\frac{\epsilon_{0}^{u} + \epsilon_{0}^{b}}{C^{2}} \right) (\lVert u\rVert_{L^{2}}^{2} + \lVert b\rVert_{L^{2}}^{2} ) \right) e^{2\delta t} dt \\
&+ 2 e^{2\delta t} \left( \langle u, dW_{u} \rangle + \langle b, dW_{b} \rangle \right)  
\end{split}
\end{equation} 
by Poincar$\acute{\mathrm{e}}$ inequality and (88), definition of $\delta$. We now define $\{S_{n}\}_{n \geq 1}$ where 
\begin{equation*}
S_{n} \triangleq \inf \{t > 0: \lVert u(t) \rVert_{L^{2}}^{2} + \lVert b(t) \rVert_{L^{2}}^{2} > n (\lVert u(0) \rVert_{L^{2}}^{2} + \lVert b(0) \rVert_{L^{2}}^{2}) \}
\end{equation*}
and $T \triangleq \tau_{\mathcal{C}} \wedge S_{n} \wedge t$ for any fixed $t> 0$. Integrating (90) over $[0, T]$ and taking expected value gives 
\begin{equation}
\begin{split} 
& \mathbb{E}[Y(T) + Z(T)] \\
\leq& \mathbb{E}[Y(0) + Z(0)] + (\epsilon_{0}^{u} + \epsilon_{0}^{b}) \mathbb{E}[e^{2\delta T} \int_{0}^{T} 1 - \frac{2(\lVert u\rVert_{L^{2}}^{2} + \lVert b\rVert_{L^{2}}^{2})}{C^{2}} dt ]. 
\end{split}
\end{equation} 
Thus, for $t < T$, as we have $t < \tau_{\mathcal{C}}$, we obtain 
from (91) taking $n\to\infty$ 
\begin{equation}
\mathbb{E}[Y(t\wedge \tau_{\mathcal{C}}) + Z(t\wedge \tau_{\mathcal{C}})] \leq \mathbb{E}[Y(0) + Z(0)]
\end{equation}
because $\lVert u(t) \rVert_{L^{2}}^{2} + \lVert b(t) \rVert_{L^{2}}^{2}$ is bounded and continuous in time. Therefore, 
\begin{equation*}
\begin{split} 
 \mathbb{E} [\lVert u(0) \rVert_{L^{2}}^{2} + \lVert b(0) \rVert_{L^{2}}^{2}]
\geq& \mathbb{E}[Y(t\wedge \tau_{\mathcal{C}} + Z(t\wedge \tau_{\mathcal{C}})] \\
\geq& \mathbb{E} [Y(t) + Z(t) \lvert t \leq \tau_{\mathcal{C}}] \mathbb{P}(t\leq \tau_{\mathcal{C}})
\geq e^{2\delta t} C^{2} \mathbb{P} \{t \leq \tau_{\mathcal{C}}\} 
\end{split} 
\end{equation*}
where we used (92), the definition of $\tau_{\mathcal{C}}$. This implies (89) and completes the proof of Proposition 4.5. 
\end{proof}

\begin{proposition}
Let $h > 0, U_{1}$ be an open neighborhood of the origin. Then given any initial condition $(u(0), b(0))$, 
\begin{equation*}
\mathbb{P}_{u(0), b(0)} \{(u(nh, b(nh)) \in U_{1} \text{ for infinitely many } n \} = 1. 
\end{equation*}
\end{proposition}

\begin{proof}
Again we denote $C> 0$ such that 
\begin{equation*}
\lVert u(0) \rVert_{L^{2}}^{2} + \lVert b(0) \rVert_{L^{2}}^{2} > C^{2} > \epsilon_{0}^{u} + \epsilon_{0}^{b}
\end{equation*}
and $\mathcal{C} \triangleq \mathcal{B}(C)$. By hypothesis, $U_{1}$ is open so that $\mathcal{B}_{1} \triangleq \mathcal{B}(C_{1}) \subset U_{1}$ for $C_{1}> 0$ sufficiently small. We let $\mathcal{B}_{0} \triangleq \mathcal{C}$ so that by Proposition 4.4, there exists $T_{0} = T_{0}(C, C_{1}) > 0$ such that for any $T \geq T_{0}$, there is $p^{\ast} > 0$ that satisfies 
\begin{equation}
\inf_{(u(0), b(0)) \in \mathcal{B}_{0}} \mathbb{P}_{(u(0), b(0))} \{(u(t), b(t)) \in \mathcal{B}_{1} \text{ for all } t \in [T, T+h] \} \geq p^{\ast}. 
\end{equation}
We let $\displaystyle{T \triangleq mh}$ for some $\displaystyle{m \in \mathbb{N}}$ such that $\displaystyle{T > (T_{0} + 2h)}$, set $\displaystyle{n^{\ast} \triangleq \frac{T}{h}}$ and $X_{n} \triangleq ((u(nh), b(nh))$. Thus, 
\begin{equation}
\begin{split} 
& \mathbb{P} \{X_{n+n^{\ast} -1} \in U_{1}: (u(t), b(t)) \in \mathcal{C} \text{ for some } t \in [(n-1)h, nh]\} \\
\geq& \mathbb{P} \{(u(nh + T-h), b(nh + T-h)) \in \mathcal{B}_{1}: \\
& \hspace{30mm} (u(t), b(t)) \in \mathcal{C} \text{ for some } t \in [(n-1)h, nh]\}
\geq p^{\ast}   
\end{split}
\end{equation} 
where we used that $\displaystyle{X_{n} = (u(nh), b(nh)), n^{\ast} = \frac{T}{h}}$, $\displaystyle{U_{1} \supset \mathcal{B}_{1}}$ and (93). Now we define for $k > 0$, 
\begin{equation*}
\begin{split} 
&\tau_{0} \triangleq \inf \{n \geq 1: (u(t), b(t)) \in \mathcal{C} \text{ for some } t \in [(n-1)h, nh]\},\\
& \tau_{k} \triangleq \inf \{n \geq \tau_{k-1} + (n^{\ast} +1): (u(t), b(t)) \in \mathcal{C} \text{ for some } t \in [(n-1)h, nh]\}.
\end{split}
\end{equation*} 
By Proposition 4.5, we know that 
\begin{equation*}
\mathbb{P} \{\inf \{s > 0: (u(s), b(s)) \in \mathcal{C} \} \geq t\} \leq \frac{\mathbb{E}[\lVert u(0) \rVert_{L^{2}}^{2} + \lVert b(0) \rVert_{L^{2}}^{2}]}{C^{2}} e^{-2\delta t} \to 0 
\end{equation*}
as $t\to\infty$. Hence, $\tau_{k}$ is finite for all $k$ $\mathbb{P}$-almost surely. We define 
\begin{equation*}
\#_{U_{1}} (n) \triangleq \text{ number of } k \in  [0, n] \text{ such that } X_{k} \in U_{1}.
\end{equation*} 
Then for any $n, M$ such that $M < n$, by (94) 
\begin{equation*}
\mathbb{P} \{\#_{U_{1}} (\tau_{n} + n^{\ast}) < M \} 
\leq (1- p^{\ast})^{n-M}.
\end{equation*}
As $p^{\ast} > 0$, we have $1- p^{\ast} < 1$ and hence $(1-p^{\ast})^{n-M} \to 0$ as $n\to + \infty$ which implies that $U_{1}$ is visited infinitely often $\mathbb{P}$-almost surely. This completes the proof of Proposition 4.6. 
\end{proof}

\emph{Proof of Theorem 2.1} 
We finally complete the proof of Theorem 2.1 following the approach of \cite[Theorem 1.1]{EM01}.  We denote by $P_{t}$ the transition density of the system (\ref{15}), (\ref{16}) and fix $\displaystyle{h > 0}$. As $\displaystyle{\int P_{h} (0, y) dy = 1}$, there exists some $y_{0}$ such that $P_{h}(0, y_{0}) > 0$. By Proposition 4.3, $P_{t}$ is regular and hence there exists a neighborhood of the origin, denoted by $A_{1}$, and a neighborhood of $y_{0}$, denoted by $A_{2}$, and $\delta_{0} > 0$ such that if $x \in A_{1}, y \in A_{2}$, then $P_{h}(x,y) > \delta_{0}$. 

We denote the normalized Lebesgue measure on $A_{2}$ by $m$. We show that for any measurable set $B \subset A_{2}$ such that $m(B) > 0$, denoting $X_{n} = (u(nh), b(nh))$ again, 
\begin{equation*}
\mathbb{P}_{(u(0), b(0))} \{X_{n} \in B \text{ for infinitely many } n \} = 1, 
\end{equation*}
which would imply that $X_{n}$ satisfies the Harris' condition and conclude the proof. 

We denote by $t_{n}$, the $n$-th time that $\{X_{n}\}$ entered $A_{1}$. By Proposition 4.6, each $t_{n}$ is finite with probability one. Let $\#_{B}(n)$ be the number of $k \in [0,n]$ such that $X_{k} \in B$. Then 
\begin{equation*}
\mathbb{P}_{(u(0), b(0))} \{X_{n} \in B \lvert X_{n-1} \in A_{1} \} 
= \int_{B} P_{h}(X_{n-1}, y) dy 
\geq m(B) \delta_{0} > 0,
\end{equation*}
where we used that $X_{n-1} \in A_{1}, y \in B \subset A_{2}$, and that $m(B) > 0, \delta_{0} > 0$. Now for any fixed $M, n$ such that $n > M > 0$, 
\begin{equation*}
\mathbb{P} \{\#_{B} (t_{n+1}) < M \} \leq (1- m(B)\delta_{0})^{n-M} \to 0 
\end{equation*}
as $n\to \infty$ because $m(B) \delta_{0} > 0$ so that $1- m(B) \delta_{0} < 1$. By arbitrariness of $M$, we conclude that $B$ is visited infinitely many times $\mathbb{P}$-almost surely and conclude the proof of Theorem 2.1.

\end{document}